\documentclass[12pt]{amsart}
\usepackage[all]{xy}
\SelectTips{eu}{}
\usepackage{fullpage}
\usepackage{amssymb}
\usepackage{amsmath}
\usepackage{hyperref}
\newcommand{\bF}{\mathbb F}
\newcommand{\bR}{\mathbb R}
\newcommand{\bZ}{\mathbb Z}
\newcommand{\Ext}{\mathsf{Ext}}
\newcommand{\HC}{H\!C}
\newcommand{\HH}{H\!H}
\newcommand{\Hom}{\mathsf{Hom}}
\newcommand{\Id}{\mathrm{Id}}

\newcommand{\Res}{\mathrm{Res}}
\newcommand{\sltwo}{\mathfrak{sl}(2)}
\newcommand{\tenk}{\otimes_k}
\newcommand{\ten}{\otimes}
\newcommand{\tenA}{\otimes_A}
\newcommand{\Tor}{\mathrm{Tor}}
\newcommand{\Tr}{\mathrm{Tr}}
\renewcommand{\leq}{\leqslant}
\renewcommand{\le}{\leqslant}
\renewcommand{\geq}{\geqslant}
\renewcommand{\ge}{\geqslant}

\numberwithin{equation}{section}

\theoremstyle{plain}
\newtheorem{theorem}[equation]{Theorem}
\newtheorem{lemma}[equation]{Lemma}
\newtheorem{prop}[equation]{Proposition}
\newtheorem{cor}[equation]{Corollary}

\theoremstyle{definition}
\newtheorem{defn}[equation]{Definition}

\newtheorem{hyp}[equation]{Hypothesis}
\newtheorem{eg}[equation]{Example}

\theoremstyle{remark} 
\newtheorem{remark}[equation]{Remark} 

\author{Dave Benson}
\address{Institute of Mathematics,  University of Aberdeen,
Fraser Noble Building, Aberdeen AB24 3UE
Scotland, UK}
\email{d.j.benson@abdn.ac.uk}
\author{Radha Kessar}
\address{Department of Mathematics, City, University of London EC1V 0HB,
  United Kingdom}
\email{radha.kessar.1@city.ac.uk}
\author{Markus Linckelmann}
\address{Department of Mathematics, City, University of London EC1V 0HB,
  United Kingdom}
\email{markus.linckelmann.1@city.ac.uk}
\title{On the BV structure of the Hochschild cohomology of 
finite group algebras}

\begin{document}

\begin{abstract}
We give a simple algebraic recipe for calculating the components of
the BV operator $\Delta$ on the Hochschild cohomology of a finite group
algebra with respect to the centraliser decomposition. 
We use this to investigate the properties of $\Delta$
and to make some computations for some particular finite groups.
\end{abstract}

\maketitle

\section{Introduction}

The Hochschild cohomology of a finite group is a Batalin--Vilkovisky 
(BV) algebra. This follows for example by observations of
Tradler~\cite{Tradler:2008a} from the fact that the group algebra over a 
commutative ring $k$ is a symmetric algebra. So there is a BV operator
$\Delta \colon \HH^n(kG) \to \HH^{n-1}(kG)$, which is related to
the Gerstenhaber bracket by the formula
\[ [x,y]=(-1)^{|x|}\Delta(x y) - (-1)^{|x|}\Delta(x) y
  -x \Delta(y). \]
Thus, if the cup product and the BV operator are known, so is the
Gerstenhaber bracket.

The BV operator on the Hochschild cohomology $\HH^*(kG)$ of a finite 
group $G$ over $k$ coming from the   standard symmetrising  form on $kG$ preserves the centraliser decomposition 
$$\HH^*(kG) \cong \bigoplus_{g}\ H^*(C_G(g),k)$$
under which $\Delta$ is the sum of degree $-1$ operators $\Delta_g$
on $H^*(C_G(g), k)$ (see  Section 3). Here $g$ runs over a set of representatives of the 
conjugacy classes in $G$. 
An individual component $\Delta_g$ depends
only on $g$ and $C_G(g)$ but not on $G$ itself, and hence in order to
describe $\Delta_g$ we may assume that $g$ is central in $G$. 
The following theorem gives a description of $\Delta_g$ in this
situation, and some of its properties.\footnote{We have to be careful
  to distinguish the cup product in Hochschild cohomology from 
  cup product in ordinary cohomology $H^*(C_G(g),k)$ in a component of the
centraliser decomposition. We use
  juxtaposition $xy$ to denote Hochschild cup product, and $x \cdot y$
  for cup product inside $H^*(C_G(g),k)$.}

\begin{theorem}\label{th:main}
Let $k$ be a commutative ring of coefficients.
Let $g$ be a central element in a finite group $G$, and let $z_g\colon
\bZ\times G \to G$ be the map which sends $(m,h)$ to $g^mh$.
Then 
\[ (z_g)^*\colon H^n(G,k) \to H^n(\bZ \times G,k)
\cong H^n(G,k) \oplus H^{n-1}(G,k) \]
has the form $(1,\Delta_g)$, where 
$\Delta_g\colon H^n(G,k)\to H^{n-1}(G,k)$ is the   component of the BV operator indexed by $g$.
The following properties hold for the map $\Delta_g$.
\begin{enumerate}
\item 
$\Delta_g$ is a $k$-linear derivation with respect to
multiplication in $H^*(G,k)$: for $x,y\in H^*(G,k)$ we have
\[ \Delta_g(x\cdot y) = \Delta_g(x)\cdot y+ (-1)^{|x|}x\cdot \Delta_g(y). \]
In particular, $\Delta_g$ is determined by its values on a 
set of generators for $H^*(G,k)$.
\item If $g\in Z(G)$ and $g'\in Z(G')$,
then the following diagram commutes.
\[ \xymatrix{H^i(G,k) \tenk H^j(G',k) \ar[r] 
\ar[d]_{(\Delta_g \otimes 1,1 \otimes \Delta_{g'})}& 
H^{i+j}(G\times G',k) \ar[d]^{\Delta_{(g,g')}} \\
(H^{i-1}(G,k) \tenk H^j(G',k)) \oplus (H^i(G,k) \tenk H^{j-1}(G',k))
\ar[r] & H^{i+j-1}(G\times G',k)} \]
where the horizontal maps are the K\"unneth maps and 
\[ \qquad (\Delta_g\otimes 1)(x\otimes y)=\Delta_g(x) \otimes y,\qquad 
(1\otimes \Delta_{g'})(x\otimes y) = (-1)^{|x|}x\otimes
  \Delta_{g'}(y). \]
If $k$ is a field, then $\Delta_{(g,g')}$ is determined by $\Delta_g$ and
$\Delta_{g'}$ by combining this with the K\"unneth formula.
\item If $\phi\colon G \to G'$ is a group homomorphism
sending $g\in Z(G)$ to $g'\in Z(G')$ then the following diagram
commutes.
\[ \xymatrix{H^n(G',k) \ar[r]^{\phi^*} \ar[d]_{\Delta_{g'}} &
H^n(G,k) \ar[d]^{\Delta_g} \\
H^{n-1}(G',k) \ar[r]^{\phi^*} & H^{n-1}(G,k)} \]
\item If $\rho\colon k\to k'$ is a homomorphism of commutative rings, then
the following diagram commutes.
\[ \xymatrix{H^n(G,k) \ar[r]^{\rho_*} \ar[d]_{\Delta_{g}} &
H^n(G,k') \ar[d]^{\Delta_g} \\
H^{n-1}(G,k) \ar[r]^{\rho_*} & H^{n-1}(G,k')} \]
\item If $g$ and $g'$ are elements of $Z(G)$ then 
we have $\Delta_{gg'}=\Delta_g + \Delta_{g'}$.
\item If $H$ is a subgroup of $G$ containing $g\in Z(G)$ then $\Delta_g$
commutes with the transfer map: if $x\in H^n(H,k)$ then
\[ \Tr_{H,G}(\Delta_g(x))=\Delta_g(\Tr_{H,G}(x))\in H^n(G,k). \]
\item If $g\in Z(G)$ and $k$ is a field of prime characteristic $p$, then
the map $\Delta_g$ is determined by its action on the cohomology
of a Sylow $p$-group of $G$, as follows. Let $g=g_pg_{p'}=g_{p'}g_p$
with $g_p$ a $p$-element and $g_{p'}$ a $p'$-element, and let
$P$ be a Sylow $p$-subgroup of $G$. Then $g_p\in Z(P)$ and
the following diagram commutes.
\[ \xymatrix{H^n(G,k)\ \ar@{>->}[r]^{\Res_{G,P}} \ar[d]_{\Delta_g} & 
H^n(P,k) \ar[d]^{\Delta_{g_p}} \\
H^{n-1}(G,k)\ \ar@{>->}[r]^{\Res_{G,P}} & H^{n-1}(P,k)} \]
\item In the case $k=\bF_p$, the map $\Delta_g$ commutes with the 
action of Steenrod operations on $H^*(G,\bF_p)$, and with the 
Bockstein homomorphism.
\item If $x\in H^1(G,k)$, we can regard $x$ as an element of $\Hom(G,k)$.
With this identification, we have $\Delta_g(x)=x(g)$.
\item If $x\in H^2(G,k)$ is in the image of $H^2(G,\bZ)\to H^2(G,k)$
then $\Delta_g(x)=0$.
\item If $x\in H^2(G,k)$ corresponds to a central extension
\[ 1 \to k^+ \to K \to G \to 1, \]
then for $h\in G$ we choose any inverse image $\hat h\in K$. Then
we have 
\[ \Delta_g(x)(h) = [\hat g, \hat h] \in k^+.\]
\end{enumerate}
\end{theorem}

The proof of Theorem \ref{th:main} appears below in the following
places.The  first assertion is  proved in Theorem~\ref{thm:Deltag}. 
The statements (i), (iii), (iv), (vi), (viii) are proved in 
Corollary \ref{co:derivation}, statement (ii) is in Proposition
\ref{kunneth-prop2}, the statements (v), (vii) are in 
Proposition \ref{Delta-g-additive}, statement (ix) is 
Proposition \ref{pr:deg1}, statement (x) is Proposition \ref{pr:deg2},
and statement (xi) is Theorem \ref{th:DeltaH2}. 

Part (vii) of Theorem \ref{th:main} says that if $k$ is a field of
prime characteristic $p$ we may as well suppose that $G$ is a 
$p$-group for the purpose of computation, and hence we give in 
Section \ref{ExamplesSection} a number of examples where we 
compute $\Delta_g$ in the case where $g\in Z(G)$ and $G$ is a finite 
$p$-group. Our first examples are cyclic groups, and more generally, 
abelian groups. Then we deal with dihedral, quaternion, and 
semidihedral $2$-groups. 

Following Tradler~\cite{Tradler:2008a}, the operator $\Delta_g$ is 
defined as the dual of the Connes operator $B\circ I$ in Hochschild 
homology. By work of Burghelea \cite{Burghelea:1985a}, the 
Connes exact sequence preserves the centraliser decomposition, and the 
corresponding components of the Connes exact sequence can be used to 
describe $\Delta_g$; see Theorems \ref{thm:Deltag1} and 
\ref{th:(1,Delta)} below for details.  
For our analysis of $\Delta_g$, we first provide a description of the 
Connes operator in the context of a general discrete group $G$.
Again, it only depends on $g$ and $C_G(g)$, and we give the following 
description for $g$ central in $G$. Consider the map $z_g\colon
\bZ\times G \to G$ as in the theorem above. Then
\[ (z_g)_* \colon H_n(\bZ\times G,k)\cong H_n(G,k)\oplus H_{n-1}(G,k)
\to H_n(G,k) \]
is equal to $(1,B \circ I)$. Dualising this statement to give a
statement about Hochschild cohomology requires the
use of a symmetrising form on $kG$, and we make use of
the standard one, whose value on $(g,h)$ is one if $gh=1$
and zero otherwise. 

We use Theorem \ref{th:main} to give an explicit description of the 
BV operator $\Delta$ in the Hochschild cohomology of finite groups 
over an arbitrary commutative ring $k$ in terms of ordinary cohomology, 
based on a simple general principle using homotopies in order to 
construct degree $-1$ operators in cohomology. This general principle, 
and its relation to $\Delta_g$, is expressed in the following theorem, 
which summarises the main parts of the Theorems \ref{degreeminusone1} 
and \ref{thm:Deltag2} below. By a homotopy on a chain complex (resp. 
cochain complex) we mean a graded map of degree $1$ (resp. $-1$). 

\begin{theorem} \label{main2}
Let $A$ be an algebra over a commutative ring $k$ and let $U$, $V$ be
$A$-modules. Let $z\in$ $Z(A)$ such that $z$ annihilates $U$ and $V$.
Let $(P,\delta)$ be a projective resolution of $U$.

There is a homotopy $s$ on $P$ such that the chain endomorphism
$s\circ\delta + \delta \circ s$ of $P$ is equal to multiplication by
$z$ on the components of $P$. For any such homotopy, the induced
homotopy $s^\vee$ on the cochain complex $\Hom_A(P,V)$ obtained from
applying the functor $\Hom_A(-,V)$ to $s$ is a cochain map 
$\Hom_A(P,V)\to$ $\Hom_A(P[1],V)$. In particular, upon taking 
cohomology, $s^\vee$ induces a degree $-1$ operator, denoted $D_z$, on 
$\Ext^*_A(U,V)$, and then $D_z$ is independent of the choice of $s$. 

If $G$ is a finite group, $A=kG$, $U=V=k$ and $g\in$ $Z(G)$, 
then  $\Delta_g = D_{g-1}$.  
\end{theorem}

See the Remark \ref{signconventions} below for sign conventions for
the differentials of the complexes arising in this theorem.
While the first part of Theorem \ref{main2} is a routine verification
(see the proof of Theorem \ref{degreeminusone1}), the identification
$\Delta_g=D_{g-1}$ in the case $A=kG$ in Theorem \ref{thm:Deltag2} 
requires Theorem \ref{thm:Deltag}, the proof of which is based on the 
Connes exact sequence relating relating Hochschild and 
cyclic cohomology. 
We give explicit homotopies for the bar resolution in Theorem 
\ref{thm:Deltag3}. We use this description in Section \ref{appSection} 
to give a short proof of the formula in Theorem \ref{th:main} (xi),
restated as Theorem \ref{th:DeltaH2} below, for the components of the 
BV operator in degree $2$, in terms of central extensions corresponding 
to degree $2$ elements in group cohomology. 

The theme of deciding when $\HH^1(kG)$ is a soluble Lie algebra has been
recently investigated by a number of authors \cite{Eisele/Raedschelders:2020a, 
Linckelmann/RubioyDegrassi:2020a, RubioyDegrassi/Schroll/Solotar:preprint}.
As an application of Theorem \ref{th:main}, we add to these the
following results, which are proved as Theorems~\ref{th:Phi<Z} 
and~\ref{th:extraspecial}. We denote by $\Phi(G)$ the Frattini subgroup of
a finite group $G$.

\begin{theorem}
If $G$ is a finite $p$-group such that $|Z(G):Z(G)\cap\Phi(G)|\ge 3$
then the Lie algebra $\HH^1(kG)$ is not soluble.
\end{theorem}

\begin{theorem}
If $G$ is an extraspecial $p$-group then the Lie algebra $\HH^1(kG)$ is soluble.
\end{theorem}

\medskip\noindent
\textbf{Acknowledgement.} 
The authors would like to thank the referee for an impressively detailed and
insightful list of comments and suggestions. 
Parts of this paper were written while all three authors were in 
residence at the Mathematical Sciences Research Institute in Berkeley, 
California, during the Spring 2018 semester, supported
by the National Science Foundation under Grant No.\ DMS-1440140.
The first author would like to thank the Isaac Newton Institute for 
Mathematical Sciences for support and hospitality during the 
programme ``Groups, representations and applications: new
perspectives'' during the Spring 2020 semester, 
when part of the work on this paper was undertaken,
supported by EPSRC grant EP/R014604/1.
The first author also thanks City, University of London for its hospitality 
during part of the preparation of this paper. The second author
acknowledges support from EPSRC grant EP/T004592/1, and the
third author 
acknowledges support from EPSRC grant EP/M02525X/1.

\section{Hochschild and cyclic homology of $kG$}     

The BV-operator in string topology  on the homology of a free loop space $LX$ is 
obtained by applying homology to the rotation map $S^1\times LX \to LX$ and 
decomposing  $H_*(S^1\times LX)$ using the  K\"unneth formula together with
the fundamental class in $H_1(S^1)$. We give a similar description for group
(co-)homology. 
Background material for this section may be found in
Benson~\cite[\S\S2.11--2.15]{Benson:1998c}. Other references include 
Loday~\cite[Chapter 7]{Loday:1998a}, Burghelea \cite{Burghelea:1985a}, and
Karoubi and Villamayor \cite{Karoubi/Villamayor:1990a}.

Let $G$ be a discrete group, and $k$ a commutative ring of    
coefficients. Hochschild homology of $kG$ has a centraliser
decomposition
\[ \HH_*(kG)\cong\bigoplus_g H_*(C_G(g),k) \]
where $g$ runs over a set of representatives of the conjugacy classes of 
$G$. This decomposition is unique up to unique isomorphism. A similar 
description of cyclic homology appears 
in~\cite{Burghelea:1985a,Karoubi/Villamayor:1990a}, and was 
reinterpreted in~\cite{Benson:1998c} in terms of extended centralisers, 
as follows. If $g\in G$, we define the extended centraliser 
$\hat C_G(g)$ to be the quotient $(\bR\times C_G(g))/\bZ$, where $\bZ$ 
is embedded in $\bR\times C_G(g)$ via the group homomorphism
sending $1$ to $(1,g)$. 
Recall that for a discrete group $G$
we have $H_*(G,k)\cong H_*(BG;k)$, where $BG$ is the classifying
space of $G$. On the other hand,
we regard $\hat C_G(g)$ as a one
dimensional Lie group, so we need to use classifying
space homology, and we have
\[ \HC_*(kG)\cong \bigoplus_g H_*(B\hat C_G(g);k). \]
The Connes exact sequence
(Connes \cite{Connes:1985a} \S II.4, 
Loday and Quillen \cite{Loday/Quillen:1984a}) 
connecting Hochschild and cyclic homology
\begin{equation}\label{eq:Connes1}
 \cdots \to \HH_{n+2}(kG) \xrightarrow{I}
\HC_{n+2}(kG)  \xrightarrow{S} \HC_n(kG) 
\xrightarrow{B} \HH_{n+1}(kG) \to \cdots 
\end{equation}
respects the centraliser decomposition (Burghelea~\cite{Burghelea:1985a}), 
and may be described as follows.
Applying the classifying space construction to the 
short exact sequence of Lie groups
\[ 1 \to \bZ \to \bR \times C_G(g) \to \hat C_G(g) \to 1 \]
we obtain a fibration sequence
\begin{equation}\label{eq:fibration} 
S^1 \to BC_G(g) \to B\hat C_G(g),
\end{equation}
where $\bR/\bZ=S^1=B\bZ$.
The (co-)homology of $S^1$ is $\bZ$ in degree $0$ and $1$ and vanishes 
in all other degrees. In particular, $H_1(S^1;\bZ)=$ 
$H_1(B\bZ;\bZ)\cong \bZ$, and we write $\nu$ for the element of 
$H_1(S^1;\bZ)$ corresponding to $1\in\bZ$. 
Dually, we write $\mu$ for the element of
$H^1(S^1;\bZ)=H^1(B\bZ;\bZ)\cong \Hom(\bZ,\bZ)\cong \bZ $
representing the identity element.
We use the same letters $\mu$ and $\nu$ for their images in $H^1(\bZ,k)$ 
and $H_1(\bZ,k)$. 

The K\"unneth formula yields a canonical identification 
\[ H_n(\bZ\times G,k)= H_0(\bZ,k) \tenk H_n(G,k) \ \oplus\  
H_1(\bZ,k)\tenk H_{n-1}(G,k) \] 
(see Weibel~\cite{Weibel:1994a}, Theorem~3.6.1), where the
summand $H_0(\bZ,k)\tenk H_n(G,k)$ is equal to the image of 
$H_n(G,k)$ under the map induced by the canonical inclusion $G\to$ 
$\bZ\times G$.  Using $\nu$ yields an identification 
\[ H_n(\bZ\times G,k) = H_n(G,k) \oplus H_{n-1}(G,k). \]
Similarly, using $\mu$ yields an identification
\[ H^n(\bZ\times G,k)= H^n(G,k) \oplus H^{n-1}(G,k), \]
where $H^n(G,k)$ is indentified to its image via the map induced by the 
canonical projection $\bZ\times G\to$ $G$. 
See Proposition \ref{prop:HZG} for more details on this identification.

The Serre spectral sequence of the fibration~\eqref{eq:fibration} has two 
non-vanishing rows, and therefore induces a long exact sequence
\begin{multline}\label{eq:Gysin}
\cdots \to H_{n+2}(BC_G(g);k) \xrightarrow{I} H_{n+2}(B\hat C_G(g);k) \\
\xrightarrow{S} H_n(B\hat C_G(g);k) \xrightarrow{B} H_{n+1}(BC_G(g);k)
\to \cdots 
\end{multline}
where we have used $\nu$ to identify $E^2_{*,1}$ with
$H_*(B\hat C_G(g);k)$. Note that the map $I$ is induced by the
inclusion $BC_G(g) \to B\hat C_G(g)$.

It was observed by Burghelea~\cite{Burghelea:1985a} that the Connes 
sequence is the direct sum of these sequences. Note that the maps in 
these sequences depend only on $g$ and $C_G(g)$, but not on $G$ itself. 
So in the following theorem, we assume that $g$ is central in $G$, and 
we identify $H_n(\bZ\times G,k)$ with $H_n(G,k) \oplus H_{n-1}(G,k)$ as 
described above, using the element $\nu\in H_1(\bZ,k)$. 

\begin{theorem}\label{thm:Deltag1}
Let $G$ be a discrete group with a central element $g\in Z(G)$. Consider 
the group homomorphism $z_g\colon \bZ\times G \to G$ which sends $(m,h)$ 
to $g^{m}h$.  The induced map 
\[ (z_g)_*\colon H_n(\bZ\times G,k)=
H_n(G,k) \oplus H_{n-1}(G,k) \to H_n(G,k) \]
has the form $\left(\begin{smallmatrix}1\\ \psi\end{smallmatrix}\right)$ 
where $\psi$ is equal to the composite
\[ H_{n-1}(G,k) \xrightarrow{I} H_{n-1}(\hat G,k) 
\xrightarrow{B} H_n(G,k). \]
\end{theorem}

\begin{proof}
The map $G\to \bZ\times G$ sending $h$ to $(0,h)$ is a section both of 
$z_g$ and of the canonical projection $\bZ\times G\to G$, so $(z_g)_*$ 
has the form 
$\left(\begin{smallmatrix}1\\ \psi\end{smallmatrix}\right)$.

Consider the diagram of groups
\[ \xymatrix{1\ar[r] &\bZ \ar[r] \ar@{=}[d] & 
\bZ\times G \ar[r] \ar[d] & G \ar[r] \ar[d] & 1 \\
1 \ar[r] & \bZ \ar[r] & \bR\times G \ar[r] & 
\hat G \ar[r] & 1 .}  \]
Here, the upper sequence is just the direct product sequence, and the 
lower sequence is the defining sequence for $\hat G$. The middle
vertical map sends $(n,h)$ to $(n,z_g(n,h))=(n,g^{n}h)$, while the right 
hand vertical map is the inclusion. 
It is easy to check that the two squares commute.
Taking classifying spaces, we have a map of fibrations
\[ \xymatrix{S^1 \ar[r] \ar@{=}[d] & S^1\times BG \ar[r] \ar[d]^{(z_g)_*} 
& BG \ar[d] \\
S^1 \ar[r] & BG \ar[r] & B\hat G ,} \]
and hence a map of long exact sequences
\[ \xymatrix{\cdots\ar[r]^(0.3){0} &H_{n-1}(BG;k) 
\ar[r]^(0.35){\left(\begin{smallmatrix}0\\1\end{smallmatrix}\right)} 
\ar[d]^I & H_{n}(BG;k) \oplus H_{n-1}(BG;k) \ar[r]^(0.65){(1,0)} 
\ar[d]^{(1, \psi)}& H_{n}(BG;k) \ar[d]^I  \ar[r]^(0.6){0} & \cdots \\
\cdots \ar[r] & H_{n-1}(B\hat G;k) \ar[r]^B & H_{n}(BG;k) \ar[r]^I & 
H_{n}(B\hat G;k) \ar[r] & \cdots} \]
The commutativity of the left hand square proves the theorem.
\end{proof}

\section{The BV operator $\Delta$ on $\HH^*(kG)$}

For any algebra $\Lambda$ which is projective as a module over the 
coefficient ring $k$, to set up a duality between Hochschild homology 
and cohomology, we need to use coefficients $\Lambda$ in homology and 
$\Lambda^\vee=\Hom_k(\Lambda,k)$ in cohomology. This is because if 
$P_*$ is a projective resolution of $\Lambda$ as a 
$\Lambda$-$\Lambda$-bimodule, then
\[\Hom_{\Lambda^e}(P_*,\Lambda^\vee) \cong 
\Hom_k(\Lambda\otimes_{\Lambda^e}P_*,k). \]
Thus if $k$ is a field, then $\HH^n(\Lambda,\Lambda^\vee)\cong
(\HH_n(\Lambda,\Lambda))^\vee$,
but for example for $k=\bZ$ we have a universal coefficient sequence
\[ 0 \to \Ext^1_\bZ(\HH_{n-1}(\Lambda,\Lambda),\bZ) \to
  \HH^n(\Lambda,\Lambda^\vee) \to \Hom_\bZ(\HH_n(\Lambda,\Lambda),\bZ)
  \to 0. \]
Similarly, cyclic cochains on $\Lambda$ are described in terms of the 
Hochschild cochain complex for $\Lambda$ with coefficients in 
$\Lambda^\vee$, and are dual to cyclic chains described in terms of 
the Hochschild chain complex for $\Lambda$ with coefficients in 
$\Lambda$. The Connes sequence for cyclic cohomology takes the form
\begin{equation}\label{eq:Connes2}
\cdots \to \HH^{n+1}(\Lambda,\Lambda^\vee) \xrightarrow{B}
\HC^n(\Lambda) \xrightarrow{S} \HC^{n+2}(\Lambda) 
\xrightarrow{I} \HH^n(\Lambda,\Lambda^\vee) \to \cdots.
\end{equation}

For $g$ an element in a discrete group $G$, the Serre spectral 
sequence in cohomology of the fibration~\eqref{eq:fibration} induces 
a long exact sequence
\[ \cdots \to H^{n+1}(BC_G(g);k) \xrightarrow{B}
H^n(B\hat C_G(g);k) \xrightarrow{S}
H^{n+2}(B\hat C_G(g);k) \xrightarrow{I}
H^{n+2}(BC_G(g);k) \to \cdots \]
where we have used the class $\mu\in$ $H^1(\bZ,k)$ introduced in the previous
section  to identify $E_2^{*,1}$ with
$H^*(B\hat C_G(g);k)$. Note that the map $I$ is induced by 
$BC_G(g) \to B\hat C_G(g)$.

\begin{theorem}\label{thm:HH*Connes}
Let $G$ be a discrete group and $k$ a commutative ring of
coefficients. Then we have a centraliser decomposition
\[ \HH^*(kG,kG^\vee) \cong \prod_g H^*(C_G(g),k), \]
where $g$ runs over a set of representatives of the conjugacy
classes in $G$. The Connes sequence is a direct product of the 
sequences
\[ \cdots \to H^{n+1}(BC_G(g);k) \xrightarrow{B}
H^n(B\hat C_G(g);k) \xrightarrow{S}
H^{n+2}(B\hat C_G(g);k) \xrightarrow{I}
H^{n+2}(BC_G(g);k) \to \cdots \]
\end{theorem}

\begin{proof}
For the centraliser decomposition see~\cite[Theorem~2.11.2]{Benson:1998c}.
The statement on the  Connes sequence is dual to Burghelea~\cite{Burghelea:1985a}, 
with essentially the same proof. See also Sections~2.11 to~2.15 of~\cite{Benson:1998c}.
\end{proof}

If $\Lambda$ is a symmetric algebra over $k$, finitely generated and 
projective as a $k$-module, then $\Lambda^\vee$ is isomorphic to 
$\Lambda$ as a $\Lambda$-$\Lambda$-bimodule, but the isomorphism 
depends on the choice of symmetrising form.

From now on, we assume that $G$ is a finite group.
Then for $kG$, there is a canonical symmetrising form. This is the 
bilinear pairing $kG \tenk kG \to k$ which sends $g\otimes g'$ to 
$1\in k$ if $gg'=1$ and zero otherwise. Using this bilinear pairing, 
we obtain an isomorphism between $kG^\vee$ and $kG$, and hence between
$\HH^*(kG,kG^\vee)$ and $\HH^*(kG,kG)$. From now on, we write 
$\HH^*(kG)$ for $\HH^*(kG,kG)$. This is a graded commutative ring, 
whose product structure was elucidated by 
Siegel and Witherspoon~\cite{Siegel/Witherspoon:1999a}.

\begin{defn}[{cf.  \cite{Burghelea:1985a},  \cite{Tradler:2008a}}] \label{def:Delta}
For $G$ finite, we define the operator 
$\Delta\colon \HH^n(kG)\to \HH^{n-1}(kG)$
to be the map induced by $I\circ B$ under the isomorphism
between $\HH^*(kG)$ and $\HH^*(kG,kG^\vee)$ given
by the symmetrising form on $kG$ described above.
\end{defn}

\begin{theorem}\label{th:(1,Delta)}
For $G$ a finite group, we have a centraliser decomposition
\[ \HH^*(kG) \cong \bigoplus_g H^*(C_G(g),k). \]
The map $\Delta$ preserves the centraliser decomposition
of $\HH^*(kG)$, and is the sum of the maps $\Delta_g=I\circ B\colon
H^*(C_G(g),k)\to H^{*-1}(C_G(g),k)$.
\end{theorem}

\begin{proof}
The centraliser decomposition in the finite case 
comes from Theorem~\ref{thm:HH*Connes}, 
see also Siegel and Witherspoon~\cite{Siegel/Witherspoon:1999a}.
Also by Theorem \ref{thm:HH*Connes}, the Connes sequence is
a direct sum of the Connes sequences for each centraliser, and
hence the components $\Delta_g$ of $\Delta$ are as stated.
\end{proof}

\begin{defn}
A \emph{Gerstenhaber algebra} is a graded $k$-algebra $R$
which is associative and graded commutative:
\[ yx = (-1)^{|x||y|}xy \]
and has a Lie bracket $[x,y]$ of degree $-1$ which satisfies
\begin{enumerate}
\item[\rm (i)] Anticommutativity:
\[ [y,x]=- (-1)^{(|x|-1)(|y|-1)}[x,y] \]
\item[\rm (ii)] The graded Jacobi identity holds:
\[ [[x,y],z]] = [x,[y,z]] - (-1)^{(|x|-1)(|y|-1)}[y,[x,z]] \]
\item[\rm (iii)] The bracket is a derivation with respect to the
  product (Leibniz identity):
\[ [x,yz] =[x,y]z+(-1)^{(|x|-1)|y|}y[x,z] \]
\end{enumerate}
\end{defn}

\begin{defn}
A \emph{Batalin--Vilkovisky algebra} (or BV algebra) is a Gerstenhaber 
algebra together with an operator $\Delta$ of degree $-1$ satisfying 
$\Delta \circ \Delta =0$ and 
\[ [x,y]=(-1)^{|x|} \Delta(xy) -(-1)^{|x|}(\Delta x)y - x(\Delta  y). \]
Thus the multiplication and the BV operator $\Delta$ determine the Lie 
bracket.
\end{defn}

For background on BV algebra 
structures in Hochschild cohomology and relationship with loop space 
topology, see Abbaspour \cite{Abbaspour:2015a}. It is pointed out
in Remark 5.1 of Rubio y Degrassi~\cite{RubioyDegrassi:2017a} 
that the restricted Lie algebra structure cannot in general be read off from the
BV-algebra structure. 

\begin{prop}[{Tradler~\cite{Tradler:2008a},  Menichi~\cite{Menichi:2004a}}]
The operator $\Delta$ of Definition~\ref{def:Delta}  
defines a BV operator on $\HH^*(kG)$
making it a BV algebra in which the Lie bracket $[-,-]$ is 
the Gerstenhaber bracket in Hochschild cohomology.
\end{prop}

\begin{remark}
Tradler gives the following formula for $\Delta$ at the level of 
Hochschild cochains, on an algebra $\Lambda$ with a symmetric, 
invariant, non-degenerate bilinear form  
$\langle -,-\rangle \colon \Lambda \times \Lambda \to k$. 
For $f\in C^n(\Lambda,\Lambda)$, define
$\Delta f\in C^{n-1}(\Lambda,\Lambda)$ by
\[ \langle \Delta f(a_1,\dots,a_{n-1}),a_n\rangle =
\sum_{i=1}^n (-1)^{i(n-1)} \langle
f(a_i,\dots,a_{n-1},a_n,a_1,\dots,a_{i-1}),1\rangle. \]
Note that this formula depends on the choice of the above
symmetrizing form. 

Explicit calculations of the BV structure on Hochschild cohomology
of finite groups have been made in a number of
different cases, see the references in Section~\ref{ExamplesSection}.
Also relevant are the papers of Le and Zhou 
\cite{Le/Zhou:2014a}, and Volkov \cite{Volkov:2016a}.
For $k$ a field, Liu and Zhou \cite{Liu/Zhou:2016a} have given an 
explicit description of the BV operator on $\HH^*(kG)$ in terms 
of Hochschild cochains, combining the centraliser decomposition and
Tradler's description of the BV operator in \cite{Tradler:2008a}.  
\end{remark}

\section{The map $\Delta_g$}

In this section, we dualise the construction in
Theorem~\ref{thm:Deltag1}  and use it to describe the components  
\[ \Delta_g\colon H^n(G,k) \to H^{n-1}(G,k) \]
of  the BV operator $\Delta$, where $G$ is a finite group.
As mentioned previously, for $g\in Z(G)$, the long exact sequence in 
cohomology for the fibration  
\[ S^1 \to BG \to B\hat G \]
has the form
\begin{equation}\label{eq:Gysin2}
\cdots \to H^{n+1}(BG;k) \xrightarrow{B} H^n(B\hat G;k) 
\xrightarrow{S} H^{n+2}(B\hat G;k) \xrightarrow{I} H^{n+2}(BG;k) \to
\cdots
\end{equation}
where we have used the class $\mu\in$ $H^1(\bZ,k)$ introduced in Section 2
 to identify $E_2^{*,1}$ with 
$H^*(B\hat G;k)$. Again, the map $I$ is induced by $BG \to B\hat G$.
The following theorem describes the map 
$\Delta_g\colon H^n(G,k) \to H^{n-1}(G,k)$
coming from the element $g\in Z(G)$.

\begin{theorem}\label{thm:Deltag}
Let $G$ be a finite group and let $g\in Z(G)$.
Consider the group homomorphism $z_g\colon \bZ\times
G \to G$ which sends $(m,h)$ to $g^mh$.  The induced map
\[ (z_g)^*\colon H^n(G,k) \to H^n(\bZ\times G,k)= H^n(G,k) \oplus
  H^{n-1}(G,k) \]
has the form $(1,\Delta_g)$, where $\Delta_g=I\circ B$ is equal to the
composite
\[ H^n(G,k)\xrightarrow{B} H^{n-1}(\hat G,k) \xrightarrow{I}
  H^{n-1}(G,k). \]
\end{theorem}
\begin{proof}
The proof is dual to the proof of Theorem~\ref{thm:Deltag1}.
\end{proof}

The following corollary proves the properties (i), (iii), (iv), 
(vi), and (viii) in Theorem~\ref{th:main}.

\begin{cor}\label{co:derivation}
Let $G$ be a finite group and let $g\in Z(G)$. Then 
\begin{enumerate}
\item The map
\[ \Delta_g\colon H^*(G,k) \to H^{*-1}(G,k) \]
is a derivation with respect to the ordinary cohomology cup product.
In particular, $\Delta_g$ is determined by its values on a
generating set of $H^*(G, k)$. 
\item The map $\Delta_g$ is natural with respect to group
  homomorphisms, and with respect to homomorphisms of
coefficient rings.
\item If $H$ is a subgroup of $G$ containing $g$ then $\Delta_g$
commutes with the transfer map and the restriction map.
\item
In the case $k=\bF_p$, the map $\Delta_g$ commutes with the Steenrod 
operations and with the Bockstein homomorphism.
\end{enumerate}
\end{cor}

\begin{proof}
The map 
\[ (z_g)^*\colon H^*(BG,k) \to H^*(S^1 \times BG,k)=
H^*(S^1,k) \otimes_k H^*(BG,k) \]
is a $k$-algebra homomorphism.
Let $\mu$ be the generator for $H^1(S^1,k)$ chosen so that
for $x\in H^*(BG,k)$ we have
\[ (z_g)^*(x)=1 \otimes x + \mu \otimes \Delta_g(x). \]
Then 
\begin{align*} 
(z_g)^*(x\cdot y)&= 
(1\otimes x+\mu\otimes\Delta_g(x))\cdot 
(1\otimes y+\mu\otimes \Delta_g(y))\\
&=1\otimes x\cdot y + \mu\otimes 
(\Delta_g(x)\cdot y+(-1)^{|x|}x\cdot\Delta_g(y)).
\end{align*}
Examining the coefficient of $\mu$, we see that
\[ \Delta_g(x\cdot y) = \Delta_g(x)\cdot y + (-1)^{|x|} x \cdot
  \Delta_g(y). \]
This proves (i). The naturality statements (ii)--(iv) follow directly 
from the fact that $(z_g)^*=(1,\Delta_g)$.
\end{proof}

For degree one elements, it is easy to describe $\Delta_g$; this is
statement (ix) in Theorem \ref{th:main}.

\begin{prop}\label{pr:deg1}
Let $G$ be a finite group, and let $g\in Z(G)$. Then identifying 
$H^1(G,k)$ with $\Hom(G,k)$, we have 
$$\Delta_g(x)=x(g)$$ 
for any $x\in H^1(G,k)$.
\end{prop}

\begin{proof}
The composite $\bZ\times G\xrightarrow{z_g}G\xrightarrow{x}k$
is equal to $1\otimes x + \mu \otimes x(g)$.
\end{proof}

\begin{remark}
It follows from Proposition \ref{pr:deg1} that if $g$ is in the  subgroup 
generated by commutators and $p$-th powers of elements in $G$,
then $\Delta_g$ is identically zero on $H^1(G,k)$.
\end{remark}

The next result is statement (x) in Theorem \ref{th:main}. 

\begin{prop}\label{pr:deg2}
Let $G$ be a finite group and let $g\in Z(G)$. Then the map $\Delta_g$ 
vanishes on the image of 
\[  H^2(G,\bZ)\to H^2(G,k). \]
\end{prop}

\begin{proof}
Consider the commutative diagram
\[ \xymatrix{H^2(G,\bZ) \ar[r]^{\Delta_g} \ar[d] & H^1(G,\bZ) \ar[d] \\
H^2(G,k) \ar[r]^{\Delta_g} & H^1(G,k) . } \]
The proposition now follows from the fact that since $G$ is
finite, we have $H^1(G,\bZ)=\Hom(G,\bZ)=0$.
\end{proof}

\section{Degree $-1$ operators on $\Ext^*_A(U,V)$}

We describe an elementary  construction principle for degree $-1$
operators on $\Ext^*_A(U,V)$ determined by a central element in an 
algebra $A$ which annihilates both modules $U$ and $V$. We use this in 
the next section to interpret the BV operator in terms of this 
construction principle.  Let $k$ be a commutative ring.

\begin{theorem} \label{degreeminusone1}
Let $A$ be a $k$-algebra, let $z\in$ $Z(A)$,
and let $U$, $V$ be $A$-modules. Suppose that $z$ annihilates both $U$ 
and $V$. Let $P=$ $(P_n)_{n\geq 0}$ together with a surjective
$A$-homomorphism $\pi \colon P_0\to U$ be a projective resolution of $U$, 
with differential $\delta=$ $(\delta_n \colon P_n\to P_{n-1})_{n\geq 1}$. 
For notational convenience, set $P_i=0$ for $i<0$ and $\delta_i=0$
for $i\leq 0$. Then the following hold.

\begin{enumerate}
\item[\rm (i)]
There is a graded $A$-homomorphism $s \colon P\to P$ of degree $1$ 
such that the chain endomorphism $\delta\circ s + s\circ \delta$ of $P$
is equal to multiplication by $z$ on $P$. 

\item[\rm (ii)]
The graded $k$-linear map
$$s^\vee = \Hom_A(s,V) \colon \Hom_A(P,V) \to \Hom_A(P[1], V)$$
sending $f\in$ $\Hom_A(P_n,V)$ to $f\circ s\in$ $\Hom_A(P_{n-1},V)$
for all $n\geq 0$ is a homomorphism of cochain complexes. In 
particular, $s^\vee$ induces a graded $k$-linear map of degree $-1$
$$D^A_z = H^*(s^\vee) \colon \Ext_A^*(U,V) \to \Ext_A^{*-1}(U,V) . $$

\item[\rm (iii)]
The graded map $D^A_z$ is independent of the choice of the projective
resolution $P$ and of the choice of the the homotopy $s$ satisfying (i). 
In particular, we have $D^A_0 = 0$. 
\end{enumerate}
\end{theorem}

\begin{remark} \label{signconventions}
If $A$ is obvious from the context, we write $D_z$ instead of $D^A_z$.
Note that $D_z$ is a graded map which depends on $U$ and $V$.
With the notation of the Theorem, we use the following sign conventions.
For $i$ an integer, the shifted complex $P[i]$ is equal, in degree $n$, to
$P_{n-i}$, with differential $(-1)^i\delta$. The cochain complex
$\Hom_A(P,V)$ has differential in degree $n$ given by precomposing with
$(-1)^{n+1}\delta_{n+1}$. (This is consistent with the standard sign 
conventions, as described in \cite[Section 2.7]{Benson:1998b}, 
for total complexes of double complexes of the form $\Hom_A(P,Q)$, 
where $Q$ is another chain complex, modulo regarding $\Hom_A(P,V)$ as a chain 
complex.) Combining the above sign conventions for shifts and total 
complexes, we get that the chain complex $P[1]$ has differential $-\delta$
and the cochain complex $\Hom_A(P[1],V)$ has in degree $n$ the
differential sending $f\in$ $\Hom_A(P_{n-1},V)$ to 
$-(-1)^{n+1}f\circ\delta_n=$ $(-1)^n f\circ\delta_n\in$ $\Hom_A(P_n,V)$. 
The sign convention for cochain complexes of the form $\Hom_A(P,V)$ has 
no impact on the definition of the operators $D_z$, but it does matter
for the signs of Bockstein homomorphisms. Had we chosen the differential
of $\Hom_A(P,V)$ simply being given by precomposing with $\delta$, then
the Bockstein homomorphisms in Proposition \ref{D-functorial} below would 
anticommute  with the operators $D_z$. 
\end{remark}

\begin{proof}[{Proof of Theorem \ref{degreeminusone1}}]
Multiplication by $z$ annihilates $U$. Thus the chain endomorphism
of $P$ induced by multiplication with $z$ is homotopic to zero.
This proves the existence of a homotopy $s$ satisfying (i). 
Let $n$ be a nonnegative integer. 
In order to show that $s^\vee$ is a cochain map, we need to show
that the following diagram of $k$-linear maps 
$$\xymatrix{
\cdots\ar[r] & \Hom_A(P_n,V) \ar[d]^{s_{n-1}^\vee} 
\ar[rr]^{(-1)^{n+1}\delta_{n+1}^\vee} 
& & \Hom_A(P_{n+1}, V) \ar[d]^{s_n^\vee} \ar[r] & \cdots \\
\cdots\ar[r] & \Hom_A(P_{n-1},V) \ar[rr]_{(-1)^{n}\delta_{n}^\vee} 
& & \Hom_A(P_{n}, V) \ar[r] & \cdots } $$
is commutative.  
The commutativity of this diagram is equivalent to 
$$(-1)^{n+1} f\circ\delta_{n+1}\circ s_n= 
(-1)^n f\circ s_{n-1}\circ\delta_n , $$
hence to 
$$f\circ (\delta_{n+1}\circ s_n + s_{n-1}\circ\delta_n) = 0 ,$$
for all $f\in$ $\Hom_A(P_n,V)$. 
By the choice of the homotopy $s$, the left side is equal to 
$f \circ \zeta$, where $\zeta \colon P_n\to P_n$ is equal to 
multiplication by $z$. Since $f(zP_n)=$ $zf(P_n)\subseteq$ $zV=\{0\}$, 
it follows that $f\circ\zeta=0$. This shows that $s^\vee$ is a cochain 
map. Taking cohomology, $s^\vee$ induces a degree $-1$ map $D_z$ 
as stated, whence (ii).

Let $P'$ be a projective resolution of $U$, with differential $\delta'$
and quasi-isomorphism $P'\to U$ given by a map $\pi' : P'_0\to U$. 
Let $s'$ be a homotopy on $P'$ with the property that the chain 
endomorphism $\delta'\circ s' + s'\circ \delta'$ of $P'$
is equal to multiplication by $z$ on $P'$. Let $a : P\to P'$ be a chain 
homotopy equivalence lifting the identity on $U$, via the maps $\pi$ and 
$\pi'$. We need to show that the homotopies $a\circ s$ and $s'\circ a$
from $P[1]$ to $P'$ induce the same map upon applying $\Hom_A(-,V)$.
Set $t = a\circ s - s'\circ a$. We will use the same letter $\zeta$ for
the chain endomorphisms of $P$ and $P'$ given by multiplication with $z$.
Using that $a$ commutes with the differentials of $P$ and $P'$, we have
$$\delta'\circ t + t\circ \delta = 
a\circ\delta\circ s - \delta'\circ s'\circ a + a\circ s\circ\delta -
s'\circ\delta'\circ a = a\circ \zeta - \zeta\circ a = 0  .$$
Taking into account that the differential of $P[1]$ is $-\delta$, this 
implies that $t$ is in fact a chain map from $P[1]$ to $P'$, or 
equivalently, from $P$ to $P'[-1]$. The homotopy class of such a chain 
map represents an element in $\Ext^{-1}_A(U,U)=$ $\{0\}$, and hence the 
chain map $t$ is homotopic to zero. That is, there is a graded map
$u : P[1]\to$ $P'$ of degree $1$ such that $t=$ 
$\delta'\circ u-u\circ\delta$, where as before the sign comes from the 
fact that the differential of $P[1]$ is $-\delta$. Since $t$ is a chain 
map, it follows that $t^\vee =$ $\Hom_A(t,V) \colon \Hom_A(P',V)\to$ 
$\Hom_A(P[1],V)$ is a cochain map. The functor $\Hom_A(-,V)$ 
sends the homotopy $u$ to a homotopy $u^\vee$ satisfying $t^\vee=$
$u^\vee\circ\delta^\vee-\delta^\vee\circ u^\vee$. We need to adjust 
$u^\vee$ with the signs needed to compensate for the signs in the 
differentials of $\Hom_A(P',V)$ and $\Hom_A(P[1],V)$ according to the
sign convention in Remark \ref{signconventions}. More precisely, one
checks that $((-1)^{n+1}u_n^\vee)$ is the homotopy which is needed to
show that the cochain map $t^\vee $ is homotopic to zero. Thus $t^\vee$ 
induces the zero map on cohomology. This shows that the maps $s^\vee$ 
and $(s')^\vee$ induce the same map $D_z$ upon taking cohomology, 
hence in particular $D_0=0$. This  proves (iii).
\end{proof}

Let $A$ be an algebra over a commutative ring $k$. Any element $z\in$ 
$Z(A)$ induces a chain endomorphism on a projective resolution $P$ of 
an $A$-module $U$, and hence a graded linear endomorphism on 
$\Ext_A^*(U,V)$, for any two $A$-modules $U$, $V$. In this way, the 
space of graded $k$-linear endomorphisms of $\Ext_A^*(U,V)$ of any 
fixed degree becomes a module over $Z(A)$. Since multiplication by $z$ 
induces an element in the centre of the module category of $A$, it 
follows easily that this module structure does not depend on the choice 
of the projective resolution $P$, and moreover, for the same reason, it 
makes no difference whether we compose the endomorphism on $U$ given
by $z$ with  an element in $\Ext_A^*(U,V)$ or compose this element with
the endomorphism on $V$ given by $z$. 

\begin{theorem} \label{degreeminusone2}
Let $A$ be a $k$-algebra, let $U$, $V$ be 
$A$-modules, and let $y$, $z\in$ $Z(A)$. 

\begin{enumerate}
\item[\rm (i)]
Suppose that $y$ and $z$ annihilate $U$ and $V$. Then $y+z$ annihilates 
$U$ and $V$, and we have $D_{y+z} = D_y + D_z$.

\item[\rm (ii)]
Suppose that $z$ annihilates $U$ and $V$. Then $yz$ annihilates
$U$ and $V$, and we have $D_{yz} = yD_z$. 

\item[\rm (iii)]
Suppose that $y$ annihilates $V$ and that $z$ annihilates $U$ and $V$. Then we have
$D_{yz}=0$. 

\item[\rm (iv)]
Let $e$ be an idempotent in $Z(A)$ such that $e$ annihilates $U$ and 
$V$. Then $D_e=0$.
\end{enumerate}
\end{theorem}

\begin{proof}
For (i), suppose that $y$ and $z$ annihilate $U$ and $V$. Then clearly 
$y+z$ annihilates $U$ and $V$. Let $P$ be a projective resolution of 
$U$, with differential $\delta$. Let $s$, $t$ be a homotopies on $P$ 
such that $\delta\circ s + s\circ \delta$ is the chain endomorphism of 
$P$ given by multiplication with $y$, and such that 
$\delta\circ t + t\circ \delta$ is the chain endomorphism of $P$ given 
by multiplication with $z$. Then $\delta\circ (s+t) + (s+t)\circ\delta$ 
is the endomorphism of $P$ given by multiplication with $y+z$. 
Statement (i) follows.
For (ii), suppose that $z$ annihilates $U$ and $V$. Then clearly $yz$ 
annihilates $U$ and $V$.  As before, let $t$ be a homotopy on $P$ such 
that $\delta\circ t + t\circ \delta$ is the chain endomorphism of $P$ 
given by multiplication with $z$. Denote by $y\cdot t$ the homotopy 
obtained from composing 
$t$ with the endomorphism given by multiplication with $y$. Then 
$\delta\circ (y\cdot t) + (y\cdot t)\circ\delta$ is equal to 
multiplication on $P$ by $yz$, which shows that $D_{yz}=$ $yD_z$. 
Statement (iii) follows from (ii) and the fact that multiplication
by $y$ annihilates $V$, hence annihilates the space $\Ext_A^*(U,V)$. 
Since $e=e^2$, statement (iv) is a special case of (iii). 
\end{proof}

The operators $D^A_z$ are compatible with the restriction to 
subalgebras $B$ containing $z$ such that $A$ is projective as a 
$B$-module.

\begin{prop} \label{res-Prop}
Let $A$ be a $k$-algebra and let $B$ be a subalgebra of $A$ such that 
$A$ is projective as a left $B$-module.
Let $z\in$ $Z(A)\cap B$, and let $U$, $V$ be $A$-modules. Suppose that
$z$ annihilates $U$ and $V$. 
We have a commutative diagram of graded maps
$$\xymatrix{\Ext^*_A(U,V) \ar[r]^{D^A_z} \ar[d] &
\Ext^{*-1}_A(U,V) \ar[d] \\
\Ext^*_B(\Res^A_B(U),\Res^A_B(V)) \ar[r]_{D^B_z} &
\Ext^{*-1}_B(\Res^A_B(U), \Res^A_B(V))  ,
}$$
where the vertical maps are induced by the restriction to $B$.
\end{prop}

\begin{proof}
Let $P$ be a projective resolution of $U$. By the assumptions on $B$,
the restriction to $B$ of $P$ is a projective resolution of 
$\Res^A_B(U)$. Thus if $s$ is a homotopy on $P$ which defines
$D^A_z$ as in Theorem \ref{degreeminusone1}, then the restriction 
to $B$ of $s$ is the corresponding homotopy for $D^B_z$. The 
commutativity of the diagram follows immediately from the construction 
of the maps $D^A_z$ and $D^B_z$. 
\end{proof}

\begin{remark}
As pointed out by the referee, Proposition \ref{res-Prop}  holds even if $A$ is
not projective as a $B$-module. In that case, $\Res^A_B(P)$ is still a resolution
of $\Res^A_B(U)$, albeit not necessarily a projective one. But since the 
homotopy $s$ on $P$ restricts to a homotopy on $\Res^A_B(P)$, the proof
of Theorem \ref{degreeminusone1} (ii) still  yields an operator of degree $-1$
upon taking cohomology in $\Hom_B(\Res^A_B(P), \Res^A_B(V))$. 
To show that this yields a commutative diagram as stated, one
observes in the proof of Theorem \ref{degreeminusone1} (iii) that it suffices to 
assume that $P'$ is a resolution
which need not be projective but admits a homotopy $s'$ as in (i).
\end{remark}

\begin{prop} \label{ind-prop}
Let $G$ a finite group, $H$ a subgroup
of $G$ and $z\in$ $Z(kG)\cap kH$. Let $U$, $V$ be $kG$-modules, and
suppose that $z$ annihilates $U$ and $V$. We have a commutative
diagram of graded maps
$$\xymatrix{\Ext^*_{kH}(\Res^G_H(U), \Res^G_H(V)) \ar[r]^{D^{kH}_z} 
\ar[d]_{\Tr_H^G} & \Ext^{*-1}_{kH}(\Res^G_H(U), \Res^G_H(V)) 
\ar[d]^{\Tr^G_H}  \\
\Ext^*_{kG}(U, V) \ar[r]_{D^{kG}_z} & \Ext^{*-1}_B(U, V)  .
}$$
\end{prop}

\begin{proof}
Let $P$ be a projective resolution of the $kG$-module $U$, with
differential denoted $\delta$. By the assumptions on $z$ and by
Theorem \ref{degreeminusone1} there is a homotopy $s$ on $P$ such that 
$\delta\circ s + s\circ \delta$ is equal to the chain endomorphism of 
$P$ given by multiplication with $z$. The operator $D_z^{kG}$ is
induced by the map sending $f\in$ $\Hom_{kG}(P_n,V)$ to
$f\circ s_{n-1}$. Since $\Res^G_H(P)$ is a projective resolution
of $\Res^G_H(U)$, it follows that $D^{kH}_z$ is induced by the
map sending $f' \in$ $\Hom_{kH}(\Res^G_H(P_n), \Res^G_H(V))$
to $f'\circ s_{n-1}$. Since $s_{n-1}$ is a $kG$-homomorphism, we have
$\Tr^G_H(f')\circ s_{n-1} =$ $\Tr^G_H(f' \circ s_{n-1})$, proving the
result.
\end{proof}

The operators $D^A_z$ satisfy a K\"unneth formula. We suppress the
superscripts in what follows, since the central element subscripts
determine which algebra we are working in. 

\begin{prop} \label{kunneth-prop1}
Let $A$, $B$ be $k$-algebras such that $A$, $B$
are finitely generated projective as $k$-modules, let $U$, $U'$ be 
$A$-modules and $V$, $V'$ be $B$-modules, all finitely generated
projective as $k$-modules. Let $z\in$ $Z(A)$ and $w\in$ $Z(B)$ such
that $z$ annihilates $U$, $U'$ and $w$ annihilates $V$, $V'$. 

Then $z\ten 1$ and $1 \ten w$ annihilate the $A\tenk B$-modules  
$U\tenk V$ and $U'\tenk V'$, and we have a commutative diagram
{\small
$$\xymatrix{
\Ext^i_A(U,U')\tenk \Ext^j_B(V,V')  
\ar[r] \ar[d]_{(D_z \ten 1,\  1 \ten D_w)}   &
\Ext^{i+j}_{A\tenk B}(U\tenk V, U'\tenk V') 
\ar[d]^{D_{z\ten 1 + 1\ten w}} \\
(\Ext^{i-1}_A(U,U')\tenk \Ext^j_B(V,V'))  
\oplus (\Ext^i_A(U,U')\tenk \Ext^{j-1}_B(V,V'))  \ar[r]  &
\Ext^{i+j-1}_{A\tenk B}(U\tenk V, U'\tenk V') ,
}$$}%
where $i$, $j$ are nonnegative integers. Moreover, we have
$D_{z\ten w} = 0$.
In particular, if $k$ is a field, then $D_{z\ten 1 + 1\ten w}$ is
determined by $D_z$, $D_w$, combined with the K\"unneth formula.
\end{prop}

\begin{proof}
Note the following sign convention (as in statement (ii) of Theorem
\ref{th:main}): the second component $1\ten D_w$
of the left vertical map sends $\eta\ten\theta$ to 
$(-1)^i\eta\ten D_w(\theta)$, where $\eta\in$ $\Ext^i_A(U,U')$ and
$\theta\in$ $\Ext^j_B(V,V')$.

The assumptions on $z$ and $w$ imply that $z\ten 1$ and 
$1 \ten w$ annihilate the $A\tenk B$-modules $U\tenk V$ and 
$U'\tenk V'$. Let $(P, \delta)$ be a projective resolution of $U$ and
$(Q,\epsilon)$ a projective resolution of $V$. Since $A$ and $B$
are projective as $k$-modules, it follows that $P\tenk Q$
is a projective resolution of the $A\tenk B$-module $U\tenk V$. 
Note the signs in the differential $\delta\ten \epsilon$ 
of $P\tenk Q$; more precisely, the differential $\delta\ten\epsilon$
sends $u \ten v \in$ $P_i\tenk Q_j$ to 
$$(\delta_i(u)\ten v ,\  (-1)^iu \ten\epsilon_j(v)) $$
in $(P_{i-1}\tenk Q_j)\oplus (P_i\tenk Q_{j-1})$.  

We have a canonical isomorphism of cochain complexes
$$\Hom_A(P,U')\tenk \Hom_B(Q,V')
\cong \Hom_{A\tenk B}(P\tenk Q, U'\tenk V')$$
thanks to the assumptions that the involved algebras and modules
are finitely generated projective as $k$-modules.
Upon taking cohomology, this induces the horizontal maps
$$\Ext^i_A(U,U')\tenk \Ext^j_B(V,V')\to
\Ext^{i+j}_{A\tenk B}(U\tenk V, U'\tenk V')$$ 
in the statement.

Let $s$ be a homotopy on $P$ such that $s\circ\delta + \delta\circ s$ 
is equal to multiplication by $z$ on $P$. Similarly, let $t$ be a 
homotopy on $Q$ such that $t\circ \epsilon + \epsilon \circ t$ is equal 
to multiplication by $w$ on $Q$.  Define the homotopy $\sigma$ on 
$P\tenk Q$ by sending $u \ten v\in$ $P_i\tenk Q_j$ to 
$$\sigma(u\ten v) = (s_i(u)\ten v,\  (-1)^i u\ten t_j(v))$$
in $(P_{i+1}\tenk Q_i) \oplus (P_i \tenk Q_{j+1})$. 
The sign $(-1)^i$ is needed because of the above mentioned sign
convention for $1\ten D_w$. We need to show 
that this homotopy has the property that the chain endomorphism
of $P\tenk Q$ given by 
$$\sigma \circ (\delta \ten \epsilon) + (\delta\ten\epsilon)
\circ \sigma$$ 
is equal to multiplication by $z \ten 1 + 1 \ten w$ on  $P\tenk Q$.  

Let $u\ten v\in$ $P_i\tenk Q_j$. We calculate first the image of 
$u\ten v$ under $\sigma \circ (\delta\ten \epsilon)$.
The differential $\delta\ten\epsilon$
sends $u\ten v$ to the element 
$$(\delta_i(u)\ten v, (-1)^i u\ten \epsilon_j(v))$$
in $(P_{i-1}\tenk Q_j) \oplus (P_i\tenk Q_{j-1})$. 
The homotopy $\sigma$ sends this to the element
$$((-1)^{i-1}\delta_i(u)\ten t_j(v), 
s_{i-1}(\delta_i(u))\ten v +  u\ten t_{j-1}(\epsilon_j(v)),
(-1)^i s_i(u) \ten \epsilon_j(v))$$
in $(P_{i-1}\tenk Q_{j+1}) \oplus (P_i\tenk Q_j) \oplus 
(P_{i+1}\tenk Q_{j-1})$. 

We calculate next the image of $u\ten v$ under 
$(\delta\ten\epsilon) \circ \sigma$. 
The homotopy $\sigma$ sends $u \ten v$ to the element
$$(s_i(u) \ten v, (-1)^i u \ten t_j(v))\ $$
in $(P_{i+1}\tenk Q_j) \oplus (P_i\tenk Q_{j+1})$.
Applying the differential $\delta\ten \epsilon$ to this element yields
$$((-1)^{i}\delta_i(u)\ten t_j(v),\ 
\delta_{i+1}(s_i(u))\ten v + u\ten \epsilon_{j+1}(t_j(v)),
(-1)^{i+1}s_i(u) \ten \epsilon_j(v))$$
in $(P_{i-1}\tenk Q_{j+1}) \oplus (P_i\tenk Q_j) \oplus 
(P_{i+1}\tenk Q_{j-1})$. 
The sum of the images of $u\ten v$ under the two maps
$\sigma \circ (\delta\circ \epsilon)$ and
$(\delta\ten\epsilon) \circ \sigma$ is therefore equal to 
$$(0, ((z\ten 1) + (1 \ten w))(u\ten v), 0)$$
as claimed. 
The first statement follows. Since $z\ten w=$ $(z\ten 1)(1\ten w)$,
the second statement follows from Theorem \ref{degreeminusone2} (iii). 
\end{proof}

\begin{prop} \label{Danticommute}
Let $A$ be a $k$-algebra, let $z$, $w\in$ 
$Z(A)$, and let $U$, $V$ be $A$-modules. Suppose that $z$, $w$ 
annihilate $U$ and $V$. The following hold.
\begin{enumerate}
\item[{\rm (i)}] 
$D_z\circ D_z=0$.
\item[{\rm (ii)}] 
$D_w\circ D_z= -D_z\circ D_w$.
\end{enumerate}
\end{prop}

\begin{proof}
Let $P=$ $(P_n)_{n\geq 0}$ together with a surjective
$A$-homomorphism $\pi \colon P_0\to U$ be a projective resolution of $U$, 
with differential $\delta=$ $(\delta_n \colon P_n\to P_{n-1})_{n\geq 1}$. 
As above, for notational convenience, we set $P_i=0$ for $i<0$ and 
$\delta_i=0$ for $i\leq 0$. 

Let $s$  be a homotopy on $P$ such that $s\circ\delta+\delta\circ s$
is equal to multiplication by $z$ on $P$. By the construction from
Theorem  \ref{degreeminusone1}, the map $D_z$ is induced
by the map sending $f\in$ $\Hom_A(P_n,V)$ to $f\circ s$.
Thus $D_z\circ D_z$ is induced by the map sending $f\in$ $\Hom_A(P_n,V)$ 
to $f\circ s\circ s$. In order to show $D_z\circ D_z=0$, we need to
show that if $f$ is a cocycle (that is, $f\circ \delta_{n+1}=0$), then
$f\circ s\circ s = 0$. For this it suffices to show that 
the graded degree $2$ map $s\circ s$ is a chain map from 
$P[2]$ to $P$. Indeed, any such chain map is homotopic to zero (as it 
represents an element in $\Ext^{-2}_A(U,U)=0$), hence, upon applying the 
contravariant functor $\Hom_A(-,V)$, it induces a cochain map 
$\Hom_A(P,V)\to$ $\Hom_A(P[2],V)$ which is still homotopic to zero and 
which therefore induces the zero map in cohomology.

Composing the chain map $s\circ\delta+\delta\circ s$ 
with $s$ in either order yields the equations (of graded endomorphisms of $P$ of degree 
$1$)
$$s\circ s\circ \delta + s\circ\delta\circ s = s\cdot z ,$$
$$s\circ \delta\circ s + \delta\circ s\circ s = z\cdot s .$$
The right sides of the two equations are equal, since
$s$ is a (graded) $A$-homomorphism, so commutes with the action of $z$.
Taking the difference of these two equations yields therefore
$$s\circ s\circ \delta - \delta\circ s\circ s = 0 .$$
This shows that $s\circ s$ is indeed a chain map 
$P[2]\to P$, which by the previous paragraph completes the proof of (i).
Statement (ii) follows from applying (i) to $z+w$ and using Theorem
\ref{degreeminusone2} (i). 
\end{proof}

Let $A$ be a $k$-algebra, $z\in$ $Z(A)$, and let $U$, $V$ be $A$-modules
which are annihilated by $z$.
The operator $D_z$ on $\Ext_A^*(U,V)$ can also be described using
an injective resolution $(I,\epsilon)$ of $V$ instead of a projective 
resolution $(P, \delta)$ of $U$. Let $p : P\to U$ and $i : V\to I$ be
quasi-isomorphisms, where $U$, $V$ are regarded as complexes concentrated
in degree $0$. Denote by $K(A)$ the homotopy cateegory of chain complexes
of $A$-modules. By standard facts (see e.g. 
\cite[Section 2.7]{Benson:1998b} or \cite[Section 2.7]{Weibel:1994a}), the
space $\Ext^*_A(U,V)$ can be identified with any of 
$$\xymatrix{\Hom_{K(A)}(P,V[n]) \ar[r]^{\cong}  & 
\Hom_{K(A)}(P, I[n]) & \ar[l]_{\cong} \Hom_{K(A)}(U, I[n]) ,} $$
where the isomorphisms are induced by composing with $i$ and 
precomposing with $p$. We reindex complexes as chain complexes,
if necessary (so in particular, an injective resolution of $V$ is
of the form $I_0\to I_{-1}\to I_{-2}\to\cdots$).

\begin{theorem} \label{degreeminusone3}
Let $A$ be a $k$-algebra, let $z\in$ $Z(A)$, and let $U$, $V$ be 
$A$-modules which are both annihilated by $z$. Let $(I,\epsilon)$
be an injective resolution of $V$ with quasi-isomorphism $i : V\to I$. 
Let $t$ be a homotopy on $I$ such
that $\epsilon\circ t+t\circ\epsilon$ is equal to multiplication by
$z$ on the terms of $I$. The graded $k$-linear map $t_\vee : 
\Hom_A(U,I) \to \Hom_A(U[1],I)$ sending $g\in$ $\Hom_A(U,I_{-n})$ to
$(-1)^n t_{-n}\circ g$ is a chain map, and the induced map in
cohomology is equal to $D_z$, where we identify the cohomology of
$\Hom_A(U,I)$ and $\Ext^*_A(U,V)$ using the isomorphisms preceding the
statement.   
\end{theorem}

\begin{proof}
Since $z$ annihilates $V$, the existence of a homotopy $t$ on $I$ 
such that $\epsilon\circ t+t\circ\epsilon$ is equal to multiplication by
$z$ is obvious. The verification that the assignment $g\mapsto (-1)^n
t\circ g$ is a chain map is analogous to the first part of the proof
of Theorem \ref{degreeminusone1}. (The sign $(-1)^n$ comes from the
fact that $g$ is regarded as a chain map $U\to$ $I[n]$, and since the 
differential of $I[n]$ is $(-1)^n\epsilon$, one needs to use the 
homotopy $(-1)^n t$ in order to obtain multiplication by $z$ as the chain 
map determined by this homotopy on $I[n]$.) 

Let $n$ be a non-negative integer.
Let $g\in$ $\Hom_A(U,I_{-n})$ be a cocycle; that is, $\epsilon\circ g=0$.
Note that this is equivalent to stating that $g : U\to I[n]$ is a 
chain map, and hence $\tilde g= g\circ p : P\to I[n]$ is a chain map, where as before
$P$ is a projective resolution of $U$ with differential $\delta$ and quasi-isomorphism
$p : P\to U$. 
Similarly, let $f\in$ $\Hom_A(P_n,V)$ be a cocycle; that is, 
$f\circ \delta=0$. As before, this means that $f : P\to V[n]$ is a chain 
map, and hence $\tilde f = i[n]\circ f : P\to I[n]$ is a chain map.

Assume now that $f$ and $g$ represent the same class in $\Ext^n_A(U,V)$.
This is equivalent to requiring that the chain maps $\tilde f$, $\tilde g$
from $P$ to $I[n]$ are homotopic. 
Thus there is a homotopy $u$ from $P$ to $I[n]$ such that
$$\tilde f - \tilde g = u\circ \delta + (-1)^n\epsilon\circ u , $$
where the sign $(-1)^n$ comes from the fact that the differential of
$I[n]$ is $(-1)^n\epsilon$. 

Let $s$ be a homotopy on $P$ such that $\delta\circ s+s\circ \delta$ is
equal to multiplication by $z$ on $P$. By the construction of $D_z$, the
image of the class of $f$ under $D_z$ is represented by $f\circ s$.
Note that $f\circ s : P\to V[n]$ is a chain map (this was noted
aleady in the proof of $D_z$ being well-defined: since $z$ annihilates
$V$, we have $0=f\cdot z=$ $f\circ\delta\circ s + f\circ s\circ\delta=$
$f\circ s\circ\delta$, where we use the assumption $f\circ\delta=0$,
and hence $f\circ s\circ\delta=0$). We need to show that $f\circ s$ and
$(-1)^n t\circ g$ represent the same class in $\Ext^{n-1}_A(U,V)$.
That is, we need to show that the chain maps $\tilde f \circ s$ and
$(-1)^nt\circ \tilde g$ from $P$ to $I[n-1]$ are homotopic, or equivalently, 
we need to show that their difference $\tilde f\circ s-(-1)^nt\circ \tilde g$
is homotopic to zero. 

Note that $f : P\to V[n]$, and hence also $\tilde f =$ $i[n] \circ f : 
P\to I[n]$, is a chain map which is zero in all degrees other than 
$n$. Since $I[n]$ is zero in all degrees bigger than $n$, it follows 
that $t\circ \tilde f=0$. Similarly, we have $\tilde g\circ s=0$. It 
follows that
\begin{align*}
\tilde f\circ s - (-1)^{n-1} t\circ\tilde g &= 
(\tilde f - \tilde g) \circ s + (-1)^{n-1} t\circ(\tilde f-\tilde g)\\
&=u\circ\delta\circ s + (-1)^n\epsilon\circ u \circ s + 
(-1)^{n-1} t\circ u\circ \delta + (-1)^{2n-1}t\circ\epsilon \circ u\\
&=u\circ\delta\circ s + (-1)^n\epsilon\circ u \circ s + 
(-1)^{n-1} t\circ u\circ \delta - t\circ\epsilon \circ u.
\end{align*}
Since $u\cdot z= u\circ (\delta\circ s+s\circ\delta)$ we have
$u\circ\delta\circ s= u\cdot z - u\circ s\circ\delta$. Similarly,
we have $t\circ\epsilon\circ u = z\cdot u - \epsilon\circ t\circ u$.
Inserting these two equations into the displayed equality and
cancelling $u\cdot z=$ $z\cdot u$ yields the expression
$$- (u\circ s\circ \delta + (-1)^{n-1} \epsilon\circ u \circ s) +
(-1)^{n-1} (t\circ u\circ\delta + (-1)^{n-1}\epsilon \circ t\circ u) .$$
The two summands in this equation are contractible chain maps
from $P$ to $I[n-1]$, via the homotopies $u\circ s$ and $t\circ u$,
respectively. This shows the result.
\end{proof}

\begin{prop} \label{D-functorial}
Let $A$ be a $k$-algebra, let $z\in$ $Z(A)$, and let $n$ be a
nonnegative integer. 

\begin{enumerate}
\item[{\rm (i)}]
The map $D_z : \Ext^n_A(U,V)\to$ $\Ext^{n-1}_A(U,V)$ is functorial 
in $A$-modules $U$ and in $V$ which are annihilated by $z$. 

\item[{\rm (ii)}]
For any $A$-module $U$ and any short exact sequence of $A$-modules 
$$\xymatrix{0\ar[r] & V \ar[r] & W \ar[r] & X \ar[r] & 0}$$
such that $z$ annihilates $U$, $V$, $W$, $X$, the map $D_z$
commutes with the connecting homomorphisms $\gamma^n :
\Ext^n_A(U,X)\to$ $\Ext^{n+1}_A(U, V)$; that is, we have
$$\gamma^{n-1}\circ D_z =  D_z \circ\gamma^n \colon 
\Ext^n_A(U,X)\to \Ext^n_A(U,V)\ .$$

\item[{\rm (iii)}]
For any short exact sequence of $A$-modules
$$\xymatrix{0\ar[r] & U \ar[r] & V \ar[r] & W \ar[r] & 0}$$
and any $A$-module $X$ such that $z$ annihilates $U$, $V$, $W$, $X$,
the map $D_z$ commutes with the connecting homomorphisms $\gamma^n :
\Ext^n_A(U,X)\to$ $\Ext^{n+1}_A(W, X)$.
that is, we have
$$\gamma^{n-1}\circ D_z =  D_z \circ\gamma^n \colon 
\Ext^n_A(U,X)\to \Ext^n_A(W,X) .$$

\end{enumerate} 
\end{prop}

\begin{proof}
Let $P$ be a projective resolution of $U$.
Since the operator $D_z$ on $\Ext_A^n(U,V)$ is induced by precomposing 
$f\in$ $\Hom_A(P_n,V)$ with a homotopy $s$ on $P$, the functoriality 
in $V$ is obvious. Using an injective resolution $I$ of $V$ and 
Theorem \ref{degreeminusone3} yields the functoriality in $U$. This 
proves (i). 

Let $U$ be an $A$-module which is annihilated by $z$, and let as 
before $P$ be a projective resolution of $U$. Let
$$\xymatrix{0\ar[r] & V \ar[r]^{i} & W \ar[r]^{p} & X \ar[r] & 0}$$
be a short exact sequence of $A$-modules $V$, $W$, $X$ which are
annihilated by $z$. Since $P$ consists of projective $A$-modules,
applying $\Hom_A(P,-)$ yields a short exact sequence of cochain 
complexes
$$\xymatrix{0\ar[r] & \Hom_A(P,V) \ar[r] & \Hom_A(P,W) \ar[r] & 
\Hom_A(P,X) \ar[r] & 0 .} $$
The connecting homomorphism $\gamma$ associated to this sequence is
constructed as follows. Let $f\in$ $\Hom_A(P_n,X)$ be a cocycle; that 
is, $f\circ\delta_{n+1}=0$. This represents a class $\underline{f}$ in 
$\Ext^n_A(U,X)$. Let $g\in$ $\Hom_A(P_n,W)$ such that $p\circ g=f$. 
Then $g\circ\delta_{n+1}$ satisfies $p\circ g\circ\delta_{n+1}=$ 
$f\circ\delta_{n+1}=0$, so $g\circ\delta_{n+1}$ factors through $i$. 
Let $h\in$ $\Hom_A(P_{n+1}, V)$ such that 
$$i\circ h = (-1)^{n+1} g\circ\delta_{n+1} .$$ 
Then $h$ is a cocycle, and by our sign conventions from Remark
\ref{signconventions} regarding the differential of $\Hom_A(P,W)$, the 
class of $h$ in $\Ext^{n+1}_A(U,V)$ is $\gamma^n(\underline{f})$.  As 
before, we suppress subscripts and superscripts to $\delta$, $s$, 
$\gamma$, and write abusively $h=$ $\gamma(f)$. Now $D_z(\underline{f})$ 
is represented by $f\circ s$, and clearly $g\circ s$ lifts $f\circ s$ 
through $p$. Therefore, by the same construction as before, applied to 
$f\circ s$, the class of $\gamma(f\circ s)$ is represented by the map $m$ 
satisfying $i\circ m=$ $(-1)^n g\circ s\circ\delta$. Since 
$s\circ\delta+\delta\circ s$ is equal to multiplication by $z$ on $P$, 
and since $z$ annihilates $W$ and commutes with $g$, it follows that  
$g\circ(s\circ\delta+\delta\circ s)=$ $0$, and hence 
$$i\circ m=(-1)^n g\circ s \circ\delta= - (-1)^n g\circ\delta\circ s=
i\circ h\circ s .$$
Since $i$ is a monomorphism, this implies $m=$ $h\circ s$. The map $m$ 
represents the class of $\gamma(D_z(f))$, and $h\circ s$ represents the 
class of $D_z(\gamma(f))$. Statement (ii) follows. A similar argument, 
using an injective resolution of $X$ and Theorem \ref{degreeminusone3}, 
yields (iii).
\end{proof}

\begin{remark} \label{TateExt-Remark}
With the notation of Theorem \ref{degreeminusone1}, if $A$ is a 
finite-dimensional selfinjective algebra over a field $k$, then the 
construction principle of degree $-1$ operators in Theorem 
\ref{degreeminusone1} extends to Tate-Ext, by replacing a projective 
resolution of $U$ with a complete resolution of $U$. More precisely, 
let $(P,\delta)$ be a complete resolution of $U$; that is, $P$ is an 
acyclic chain complex of projective $A$-modules together with an
isomorphism $\mathrm{Im}(\delta_0)\cong$ $U$. Then 
$\widehat{\Ext}^n_A(U,V) \cong$ $H^n(\Hom_A(P,V))$ for all integers $n$; 
for $n$ positive this coincides with $\Ext^n_A(U,V)$. If $z\in$ $Z(A)$ 
annihilates $U$, then multiplication by $z$ on $P$ is a chain 
endomorphism which is homotopic to zero, and thus there is a homotopy 
$s$ on $P$ such that $s\circ\delta + \delta\circ s$ is equal to 
multiplication by $z$. (In fact, for this part of the construction, it 
suffices to assume that the endomorphism of $U$ given by multiplication 
with $z$ factors through a projective module). If $z$ also annihilates 
$V$, then just as in the proof of Theorem \ref{degreeminusone1} the 
correspondence sending $f\in$ $\Hom_A(P_n,V)$ to $f\circ s$ induces for
any integer $n$ an operator $\hat D_z : \widehat{\Ext}_A^n(U,V)\to$ 
$\widehat{\Ext}_A^{n-1}(U,V)$, which for $n\geq 2$ coincides with
the operator $D_z$. For $G$ a finite group, the Tate-Hochschild
cohomology $\widehat{\HH}^*(kG)$  of $kG$ admits a centralizer decomposition
analogous to that of $\HH^*(kG)$ in terms of Tate-Ext of centralizers
of group elements, and hence the above construction yields a degree $-1$
operator on $\widehat{\HH}^*(kG)$. It would be interesting to check that this
coincides with the extension of the BV-operator to $\widehat{\HH}^*(kG)$  
in work of  Liu, Wang, and Zhou~\cite{Liu/Wang/Zhou:2021a}. 
\end{remark}

\begin{remark} \label{TorRemark}
A construction analogous to that in Theorem \ref{degreeminusone1} exists for $\Tor$.
Let $A$ be a $k$-algebra, $U$ an $A$-module with  a projective resolution $(P, \delta)$, let
$W$ be a right $A$-module and  let $z\in Z(A)$ such that $z$ annihilates $U$ and $W$.
Then there is a homotopy $s : P\to P[-1]$ such that $s\circ \delta + \delta \circ s$ is
the graded chain endomorphism of $P$ given by multiplication with $z$. Since $z$
also annihilates  $W$, it follows that $\Id_W\ten s : W\tenA P \to W\tenA P[-1]$ is
a chain map. Taking homology yields a degree $1$ operator $\Tor^A_n(W,U)\to$
$\Tor^A_{n+1}(W,U)$, for all $n\geq 0$, which satisfies the formal properties
analogous to those developed for the operators $D_z$ in this section.
A similar construction, in which $z$ is an integer, has been used  for calculating
torsion  in loop space homology in Levi~\cite{Levi:1996b}.  
\end{remark}

\section{The BV operator in terms of homotopies on projective
resolutions}

Let $k$ be a commutative ring. For $G$ a finite group and an element 
$g\in$ $Z(G)$ we denote as before for any positive integer $n$ by 
$\Delta_g = I\circ B \colon H^n(G,k)\to$ $H^{n-1}(G,k)$ the map obtained 
from the long exact sequences \eqref{eq:Connes2}. The following result
shows that $\Delta_g$ can be obtained as a special case of the
construction described in Theorem \ref{degreeminusone1}, implying in 
particular that the component $\Delta_1$ of the BV operator $\Delta$ 
on the summand $H^*(G,k)$ in the centraliser decomposition of 
$\HH^*(kG)$ corresponding to the unit element $1$ of $G$ is zero.

\begin{theorem} \label{thm:Deltag2}
Let $G$ be a finite group and $g\in$ $Z(G)$. With the notation from 
Theorem \ref{degreeminusone1}, we have $\Delta_g =$ $D_{g-1}$. In 
particular, we have $\Delta_1 = D_0 = 0$.
\end{theorem}

In order to show this, we will make use of Theorem \ref{thm:Deltag}.
As before, we denote by
$$z_g \colon \bZ \times G \to G$$
the group homomorphism sending $(n,h)$ to $g^nh$, where $n\in$ $\bZ$
and $h\in$ $G$. Note that $z_1$ is the canonical projection onto
the second component of $\bZ\times G$.
In order to describe the map induced by $z_g$ on
cohomology, we choose a projective resolution $P_G$ of $k$ as a $kG$-module and 
a projective resolution of $P_\bZ$ of $k$ as a 
$k\bZ$-module. Here $k\bZ$ is the group algebra over $k$ of the infinite
cyclic group $(\bZ, +)$. We will need to describe quasi-isomorphisms
$P_\bZ\tenk P_G\to$ $z_g^*(P_G)$ as complexes of 
$k(\bZ\times G)$-modules. 

\medskip
Identify $k\bZ=$ $k[u,u^{-1}]$ for some indeterminate $u$ via
the unique algebra isomorphism sending $1_\bZ$ to $u$. We choose for
$P_\bZ$ the two-term complex (in degrees $1$ and $0$) of the form
$$\xymatrix{k[u,u^{-1}] \ar[rr]^{u-1} & & k[u,u^{-1}] ,  }$$
where the superscript $u-1$ is the map given multiplication with $u-1$.
This is a projective resolution of $k$ as a $k[u,u^{-1}]$-module, 
together with the augmentation map $k[u, u^{-1}]\to k$ sending $u$ to 
$1$. Note that pairs consisting of an infinite cyclic group and a generator
are unique up to unique isomorphism. Choosing a generator is 
equivalent to choosing a projective resolution of the form above. 

We will make use of the following special case of the 
Tensor-Hom-adjunction. We adopt the following shorthand: for any
$kG$-module $M$ we write
$$M[u,u^{-1}] = k[u,u^{-1}]\tenk M  .$$

\begin{lemma} \label{tensorhomspecial}
With the notation above, let $M$, $N$ be $kG$-modules. We have
a natural isomorphism of $k$-modules
$$\Hom_{k(\bZ\times G)}(M[u,u^{-1}], z_g^*(N))\cong
\Hom_{kG}(M,N)$$
sending a $k(\bZ\times G)$-homomorphism $f \colon M[u,u^{-1}]\to$ 
$z_g^*(N)$ to the $kG$-homomorphism $M\to $ $N$ given by
$m\mapsto$ $f(1\otimes m)$ for all $m\in$ $M$.
\end{lemma}

Let $P_G=$ $(P_n)_{n\geq 0}$ be a projective resolution of the 
trivial $kG$-module, with differential $(\delta_n)_{n\geq 1}$.
We adopt the convention $P_{-1}=\{0\}$ and $\delta_0=0$. 
Note that the $k(\bZ\times G)$-module structure of 
$z^*_g(P_n)$ is given via $u^i \otimes h$ acting as left 
multiplication by $g^ih$ on $P_n$, where $n\geq$ $0$. 
The degree $n$ term of $P_\bZ \tenk P_G$ is equal to 
$$P_n[u,u^{-1}] \oplus P_{n-1}[u,u^{-1}]$$
for any $n\geq 0$.
Denote by $\eta =$ $(\eta_n)_{n\geq 1}$ the differential of 
$P_\bZ \tenk P_G$. Use the same letter $\delta_n$ for the obvious
extension $\Id\otimes \delta_n$ of $\delta_n$ to $P_n[u, u^{-1}]$. 
The differential $\eta$ is given in degree $n\geq 1$ by
\begin{equation} \label{eq:ZGdiff}
\xymatrix{ \eta_n  : & 
 P_n[u,u^{-1}] \oplus P_{n-1}[u,u^{-1}]  
\ar[rrr]^{\left(\begin{smallmatrix} \delta_n & u-1 \\ 
0 & -\delta_{n-1} \end{smallmatrix}\right)} 
& & & P_{n-1}[u,u^{-1}] \oplus P_{n-2}[u,u^{-1}] . } 
\end{equation} 
We describe the identification $H^n(\bZ\times G,k)=$ $H^n(G,k)\oplus
H^{n-1}(G,k)$ in Theorem \ref{thm:Deltag} as follows.

\begin{prop} \label{prop:HZG}
With the notation above, the canonical split exact sequences 
$$\xymatrix{0 \ar[r] & P_n[u,u^{-1}] \ar[r]
&  P_n[u,u^{-1}] \oplus P_{n-1}[u,u^{-1}] \ar[r]
&  P_{n-1}[u,u^{-1}]   \ar[r] & 0}$$
define a degreewise split short exact sequence of chain complexes
of $k(\bZ\times G)$-modules
$$\xymatrix{0 \ar[r] & k[u,u^{-1}]\tenk P_G \ar[r] 
& P_\bZ\tenk P_G \ar[r] & k[u,u^{-1}]\tenk P_G[1] \ar[r] & 0}\ .$$ 
Applying $\Hom_{k(\bZ\times G)}(-,k)$, with the appropriate signs for 
the differentials, yields a short exact sequence of cochain complexes 
of $k$-modules
$$\xymatrix{0 \ar[r] & \Hom_{kG}(P_G[1],k) \ar[r] 
& \Hom_{k(\bZ\times G)}(P_\bZ\tenk P_G,k) \ar[r] 
& \Hom_{kG}(P_G,k) \ar[r] & 0}$$ 
which splits canonically, and hence yields a canonical identification
$$H^n(\bZ\times G, k) = H^{n}(G,k) \oplus H^{n-1}(G,k) .$$
This is the identification in Theorem \ref{thm:Deltag}.
\end{prop}

\begin{proof}
The exact sequence of chain complexes of $k(\bZ\times G)$-modules is a 
special case of a tensor product of the chain complex $P_G$ 
with a two-term chain complex and easily verified. 
The exact sequence of cochain complexes of $k$-modules is obtained by 
applying the contraviariant functor $\Hom_{k(\bZ\times G)}(-,k)$ to the 
previous sequence and then using the canonical adjunctions 
$$\Hom_{k(\bZ\times G)}(P_n[u,u^{-1}],k) \cong \Hom_{kG}(P_n,k)$$
from Lemma \ref{tensorhomspecial}.
The fact that this sequence splits canonically follows from the
observation that multiplication by $u-1$ has image in the
kernel of any $k(\bZ\times G)$-homomorphism $P_n[u,u^{-1}] \to k$
and so the non-diagonal entry $u-1$ in the differential 
\eqref{eq:ZGdiff} of 
$P_\bZ\tenk P_G$ becomes zero upon applying the functor 
$\Hom_{k(\bZ\times G)}(-,k)$. To see that this is the identification
in Theorem  \ref{thm:Deltag}, consider the class $\mu$ in
$H^1(\bZ,k)$ corresponding to the group homomorphism $\bZ\to k$ sending 
$1_\bZ$ to $1_k$. This is the class of the $1$-cocycle (abusively still 
denoted by the same letter) $\mu : k[u,u^{-1}]\to k$ sending $u$ to 
$1$. The explicit description of the maps in the statement implies that 
upon taking cohomology in degree $n$, the map $\Hom_{kG}(P_G[1],k) \to$
$\Hom_{k(\bZ\times G)}(P_\bZ\tenk P_G,k)$ induces a map which sends 
$x\in$ $H^{n-1}(G,k)$ to the image of $\mu \ten x$ in 
$H^n(\bZ\times G,k)$, as required. 
\end{proof}

For $n\geq 0$, denote by 
$$e_n \colon P_n[u,u^{-1}] \to z_g^*(P_n)$$
the $k(\bZ\times G)$-homomorphism defined by $e_n(u^i\otimes v)=$
$g^i v$ for all $i\in$ $\bZ$ and $v\in$ $P_n$. Equivalently, $e_n$
corresponds to the identity on $P_n$ under the adjunction from
Lemma \ref{tensorhomspecial}. The following theorem parametrises
homotopies in Theorem \ref{degreeminusone1} in terms of certain
quasi-isomorphisms $P_\bZ\tenk P_G\to$ $z_g^*(P_G)$ lifting the
identity on $k$. 

\begin{theorem} \label{thm:Pqi1} 
With the notation above, for any $n\geq -1$ let 
$f_n \colon P_n[u,u^{-1}]\to$
$z_g^*(P_{n+1})$ be a $k(\bZ\times G)$-homomorphism and let
$s_n \colon P_n\to$ $P_{n+1}$ be the corresponding $kG$-homomorphism 
sending $a\in$ $P_n$ to $f_n(1\otimes a)$. The following are equivalent.
\begin{enumerate}
\item The graded $k(\bZ\times G)$-homomorphism 
$$(e_n, f_{n-1})_{n\geq 0} \colon P_\bZ\tenk P_G\to z_g^*(P_G)$$
is a quasi-isomorphism of chain complexes which lifts the identity 
on $k$.
\item The graded $kG$-homomorphism 
$$(s_{n-1})_{n\geq 0} \colon P_G[1]\to P_G$$
is a homotopy with the property that the chain map
$\delta\circ s + s\circ \delta$ is equal to the endomorphism of $P_G$ 
given by left multiplication with $g-1$.
\end{enumerate}
\end{theorem}

\begin{proof}
In degree $0$, the map $e_0$ clearly lifts the identity on $k$, so
we need to show that $(e_n, f_{n-1})_{n\geq 0}$ is a chain map
if and only if $s\circ\delta + \delta\circ s$ is equal to 
multiplication by $g-1$. By the definition of the differential of
$P_\bZ\tenk P_Q$, we have that $(e_n, f_{n-1})_{n\geq 0}$ is
a chain map if and only if for any $n\geq 1$ we have
$$(e_{n-1}, f_{n-2}) \circ
\begin{pmatrix} \delta_n & u-1 \\ 0 & -\delta_{n-1} \end{pmatrix}
 = \delta_n \circ (e_n, f_{n-1}), $$
where we use as before the same letter $\delta_n$ for the extension 
$\Id\otimes \delta_n$ of $\delta_n$ to $P_n[u,u^{-1}]$. This is 
equivalent to
$$(e_{n-1}\circ\delta_n, e_{n-1}\circ (u-1)-f_{n-2}\circ\delta_{n-1})=
(\delta_n\circ e_n, \delta_n\circ f_{n-1}) ,$$
where we have used the notation $(u-1)$ for the map given by
multiplication with $u-1$. In the first component, this holds
automatically since $(e_n)_{n\geq 0}$ is a chain map. Thus the 
previous condition is equivalent to
$$e_{n-1}\circ (u-1) = 
\delta_n\circ f_{n-1} + f_{n-2}\circ \delta_{n-1} .$$
Note that the left side is equal to $(g-1)\circ e_{n-1}$, where $(g-1)$ 
denotes the map given by multiplication with $g-1$. Through the obvious 
versions of the adjunction from Lemma \ref{tensorhomspecial}, this is 
equivalent to the statement that $\delta\circ s+s\circ\delta$ is equal 
to left multiplication by $g-1$ on $P_G$ as stated.
\end{proof}

\begin{proof}[{Proof of Theorem \ref{thm:Deltag2}}]
The fact that $s^\vee$ is a cochain map is the special case of
Theorem \ref{degreeminusone1}, applied to $A=kG$, $U=V=k$,
and $z=$ $g-1$. Thus $H^n(s^\vee)=$ $D_{g-1}$, for $n\geq$ $0$. 
By Theorem \ref{thm:Pqi1} (and with the notation of that theorem)
the graded map $(s_{n-1})_{n\geq 0} \colon P_G[1]\to $ $P_G$ induces
a quasi-isomorphism
$$(e_n, f_{n-1})_{n\geq 0} \colon P_\bZ\otimes_k P_G\to z_g^*(P_G)$$
which lifts the identity on $k$. Applying the functor
$\Hom_{k(\bZ\times G)}(-,k)$, with the appropriate signs for the
differentials, and making use of the adjunction
Lemma \ref{tensorhomspecial}, yields a cochain map
$$\Hom_{kG}(P_G,k) \to \Hom_{k(\bZ\times G)}(P_\bZ \otimes_k P_G,k) .$$
Taking cohomology, and using the identification in Proposition 
\ref{prop:HZG}, yields for any $n\geq 0$ a map 
$$H^n(G,k) \to H^n(\bZ\times G, k)=H^n(G,k)\oplus H^{n-1}(G,k) .$$
By construction, the second component of this map is induced by
$s^\vee$, hence equal to $D_{g-1}$. By Theorem \ref{th:(1,Delta)},
the second component is also equal to $\Delta_g$ as stated.
\end{proof}

The following result is Theorem \ref{th:main} (ii). 

\begin{prop} \label{kunneth-prop2}
Let $G$, $H$ be finite groups, let $g\in$ $Z(G)$ and $h\in$ $Z(H)$.
Let $i$, $j$ be nonnegative integers, let $x\in$ $H^i(G,k)$ and
$y\in$ $H^j(H,k)$. Identify $k(G\times H)\cong$ $kG\tenk kH$ and
identify $x\ten y$ with its canonical image in $H^{i+j}(G\times H,k)$.
Then $\Delta_{(g,h)}(x\ten y)$ is equal to the canonical image of 
$$\Delta_g(x) \ten y + (-1)^i x\ten \Delta_h(y)\ $$
in $H^{i+j-1}(G\times H, k)$. 
\end{prop}

\begin{proof}
We apply Proposition \ref{kunneth-prop1} to the case $A=kG$, $U=U'=k$, 
$B=kH$, $V=V'=k$, $z=g-1$, $w=h-1$. We have
$$\Delta_g(x) \ten y + (-1)^i x\ten \Delta_h(y) =
D_{g-1}(x) \ten y + (-1)^i x\ten D_{h-1}(y) . $$
By Proposition \ref{kunneth-prop1}, the image in 
$H^{i+j}(G\times H, k)$ of this element is equal to
$$D_{(g-1)\ten 1 + 1\ten (h-1)}(x \ten y) .$$
It follows from the last statement in Proposition \ref{kunneth-prop1}
that $D_{(g-1)\ten (h-1)}=0$.  Since 
$$(g-1)\ten (h-1) = 
(g\ten h)-(1\ten 1) - ((g-1)\ten 1) - (1\ten (h-1)))$$
this implies that 
$$0 = D_{(g-1)\ten (h-1)}= D_{(g \ten h)-(1\ten 1)} -
D_{(g-1)\ten 1} - D_{1\ten (h-1)}$$
or equivalently, 
$$D_{(g \ten h)-(1\ten 1)} = D_{(g-1)\ten 1} + D_{1\ten (h-1)} =
D_{(g-1)\ten 1 + 1\ten (h-1)} . $$
The left side in the last equation is $\Delta_{(g,h)}$, whence the 
result. 
\end{proof}

The following proposition implies the statements (v) and (iv) in 
Theorem \ref{th:main}.

\begin{prop} \label{Delta-g-additive}
Let $G$ be a finite group and $g$, $h\in$ $Z(G)$. 
\begin{enumerate}
\item[{\rm (i)}]
We have $\Delta_{gh} = \Delta_g+\Delta_{h}$.
\item[{\rm (ii)}]
We have $\Delta_{g^m}=m\Delta_g$ for any
positive integer $m$. 
\item[{\rm (iii)}]
If the order of $g$ is invertible in $k$, then $\Delta_g=0$. In 
particular, if $k$ is a field of prime characteristic $p$ and $g$ a 
$p'$-element in $Z(G)$, then $\Delta_g=0$.
\item[{\rm (iv)}]
If $k$ is a field of prime characteristic $p$, then $\Delta_g=$
$\Delta_{g_p}$, where $g_p$ is the $p$-part of $g$.
\item[{\rm (v)}]
If $k$ is a field of prime characteristic $p$, denoting by $P$ a
Sylow $p$-subgroup of $G$, we have $g_p\in$ $Z(P)$, and for any
positive integer $n$ the diagram 
\[ \xymatrix{H^n(G,k)\ \ar@{>->}[r]^{\Res_{G,P}} \ar[d]_{\Delta_g} & 
H^n(P,k) \ar[d]^{\Delta_{g_p}} \\
H^{n-1}(G,k)\ \ar@{>->}[r]^{\Res_{G,P}} & H^{n-1}(P,k)} \]
is commutative with injective horizontal maps.
\end{enumerate}
\end{prop}

\begin{proof}
The identity $(g-1)(h-1)=$ $(gh-1)-(g-1)-(h-1)$ implies that 
$D_{(g-1)(h-1)}=$ $D_{gh-1}-D_{g-1}-D_{h-1}$. By Theorem
\ref{degreeminusone2} (iii) we have $D_{(g-1)(h-1)}=0$, and hence 
$D_{gh-1}=$ $D_{g-1} + D_{h-1}$. By Theorem \ref{thm:Deltag2} this 
yields the equality stated in (i), and (ii) is an immediate consequence 
of (i). By Theorem \ref{thm:Deltag2} we have $\Delta_1=0$. Thus (iii) 
follows from (ii). Assume that $k$ is a field of prime characterisic 
$p$. Note that since $g$ is central, so are the $p$-part $g_p$ and the 
$p'$-part $g_{p'}$ of $g$. By (i) we have $\Delta_g=$ $\Delta_{g_p}+$ 
$\Delta_{g_{p'}}$, and by (iii) this is equal to $\Delta_{g_p}$, whence 
(iv). Since $g$, and hence also $g_p$, is central, it follows that $g_p$
is contained in any Sylow $p$-subgroup of $G$. Statement (v) follows 
from (iv), the compatibility of $\Delta_{g_p}$ with restriction from 
$G$ to $P$ by Corollary \ref{co:derivation} (iii), together with the 
standard fact that the restriction map $H^*(G,k)\to$ $H^*(P,k)$ is 
injective.
\end{proof}

\section{The BV operator on the bar resolution} \label{bar-res-section}

Let $k$ be a commutative ring, $G$ a finite group and $g\in$ $Z(G)$.
The purpose of this section is to calculate explicitly a homotopy
$s$ on the bar resolution of the trivial $kG$-module as in Theorem 
\ref{thm:Deltag2}.
This will be used in the proof of Theorem \ref{th:DeltaH2}. 
As before, we denote by $z_g \colon \bZ\times G\to$ $G$ the group 
homomorphism sending $(m,h)$ to $g^m h$.

\medskip
We choose for $P_G$ the projective resolution of the trivial $kG$-module
which in degree $n$ term is equal to $kG^{n+1}$, where $n\geq$ $0$ and 
where $G^{n+1}$ is the direct product of $n+1$ copies of $G$, with 
differential $\delta_n$ given for $n\geq$ $1$ by
$$\delta_n(a_0,a_1,..,a_n) = 
\sum_{i=1}^{n-1} (-1)^i (a_0,..,a_ia_{i+1},.., a_n)
+(-1)^n (a_0,a_1,..,a_{n-1})$$
where the $a_i$ are elements in $G$. 
The left $kG$-module structure on the terms $kG^{n+1}$ is given by left 
multiplication with $G$ on the first copy of $G$. (This is the 
resolution obtained from tensoring the Hochschild resolution of $kG$ by 
$-\otimes_{kG} k$.) In particular, in degree $1$, we have
$$\delta_1(a_0,a_1) = a_0a_1-a_0 .$$
As earlier, we identify $k\bZ=k[u,u^{-1}]$ and $k(\bZ\times G)=$ $k[u,u^{-1}]\tenk kG$. 
The $k(\bZ\times G)$-module structure of $z^*_g(kG^{n+1})$
is given by the action of $u^i \otimes h$ acting as left multiplication
by $g^i h$ on the first copy of $G$ in $G^{n+1}$. 

\begin{theorem} \label{thm:Deltag3}
With the notation above, for any $n\geq$ $1$ denote by 
$s_{n-1} \colon kG^n\to$ $kG^{n+1}$ the $kG$-homomorphism defined by
$$s_{n-1}(a_0,a_1,..,a_{n-1}) = 
\sum_{i=0}^{n-1} (-1)^i (a_0,..,a_i, g, a_{i+1},..,a_{n-1}) ,$$
where $a_i\in$ $G$ for $0\leq i\leq n-1$. Applying the functor
$\Hom_{kG}(-,k)$ to the graded map 
$$s= (s_{n-1})_{n\geq 1} \colon P_G[1] \to P_G$$
yields a map of  cochain complexes of $k$-modules
$$s^\vee \colon \Hom_{kG}(P_G,k) \to \Hom_{kG}(P_G[1],k)$$
such that, for any $n\geq$ $1$, we have
$$\Delta_g = H^n(s^\vee) \colon H^n(G,k) \to H^{n-1}(G,k) .$$
\end{theorem} 

\begin{proof}
By Theorem \ref{thm:Pqi1}, it suffices to show that $s$ is a
homotopy such that $\delta\circ s+ s\circ \delta$ is the chain
map given by multiplication with $g-1$.
This is the content of the next theorem, whence the result.
\end{proof}

As in the previous section, for $n\geq 0$, denote by 
$$e_n \colon kG^{n+1}[u,u^{-1}] \to z_g^*(kG^{n+1})$$
the $k(\bZ\times G)$-homomorphism defined by $e_n(u^m\otimes h)=$
$g^m h$ for all $m\in$ $\bZ$ and $h\in$ $G$. 
For $n\geq$ $1$ denote by 
$$ f_{n-1} \colon kG^{n}[u, u^{-1}] \to z_g^*(kG^{n+1}) $$
the $k(\bZ\times G)$-homomorphism given by
$$f_{n-1}(u^m\otimes (a_0,a_1,..,a_{n-1})) = \sum_{j=0}^{n-1} 
(-1)^j (g^{m}a_0,.., a_j, g, a_{j+1},..,a_{n-1}) , $$
with $m\in\bZ$ and the $a_j$ in $G$. Through an adjunction as in 
Lemma \ref{tensorhomspecial}, the map $f_{n-1}$ corresponds to the 
$kG$-homomorphism $s_{n-1} \colon kG^n\to$ $kG^{n+1}$ defined in the previous 
theorem.

\begin{theorem} \label{thm:Pqi2}
With the notation above, the following hold.

\begin{enumerate}

\item 
The graded $k(\bZ\times G)$-homomorphism 
$(e_n, f_{n-1})_{n\geq 0} \colon P_\bZ\tenk P_G\to z_g^*(P_G)$
is a quasi-isomorphism which lifts the identity on $k$.

\item 
The graded $kG$-homomorphism $(s_{n-1})_{n\geq 1} \colon P_G[1]\to P_G$ is a 
homotopy with the property that $\delta\circ s + s\circ \delta$ is equal
to the endomorphism of $P_G$ given by left multiplication with $g-1$.

\end{enumerate}
\end{theorem}

\begin{proof}
By Theorem \ref{thm:Pqi1}, the two statements are equivalent.
We prove the second statement. That is, for $n\geq 1$ and $x=$ 
$(a_0,a_1,..,a_{n-1}) \in$ $P_{n-1}$, we need to prove the equality
$$\delta_n(s_{n-1}(x)) = (g-1)x - s_{n-2}(\delta_{n-1}(x)) . $$
We start with the left side. We have
$$\delta_n(s_{n-1}(x)) = \sum_{j=0}^{n-1} \ 
(-1)^j \delta_n(a_0,..,a_j,g ,a_{j+1},..,a_{n-1}) .$$
We need to calculate the summands
$$\delta_n(a_0,..,a_j, g, a_{j+1},..,a_{n-1})\ .$$
The definition of $\delta_n$ yields an alternating sum over an index
$i$ running from $0$ to $n-1$, which we will need to break up according 
to wether $0\leq i<j$, $i=j$, $i=j+1$, $j+1<i\leq n-1$. We have

\begin{align*}
\delta_n(a_0,..,a_j,g,a_{j+1},..,a_{n-1}) 
 &= \sum_{i=0}^{j-1} 
(-1)^i (a_0,..,a_ia_{i+1},.., a_j, g, a_{j+1},.., a_{n-1}) \\
 &\qquad + (-1)^j (a_0,..,a_jg, a_{j+1},..,a_{n-1}) \\
 &\qquad + (-1)^{j+1} (a_0,..,a_j,ga_{j+1},..,a_{n-1}) \\
 &\qquad + \sum_{i=j+2}^{n-1} (-1)^j 
(a_0,..,a_j,g,a_{j+1},..,a_{i-1}a_i,..,a_{n-2}) \\
 &\qquad + (-1)^n(a_0,..,a_j,g, a_{j+1},..,a_{n-2}).
\end{align*}

For $j=0$ or $j=n-1$ or $j=n-2$ this formula takes a slightly different 
form. If $j=0$, then the first sum is empty (hence zero by convention). 
The fourth term (that is, the sum indexed $\sum_{i=j+2}^{n-1}$) is
empty if $j$ is one of $n-1$, $n-2$, so zero. In addition, if $j=n-1$, 
then the third term does not appear (because the component $ga_{j+1}$ 
is not defined in that case) and the last summand takes the form 
$(-1)^n(a_0,..,a_{n-1})$, and if $j=n-2$, then the last summand takes 
the form  $(-1)^n(a_0,..,a_{n-2}, g)$. It follows that

\begin{align*}
\delta_n(s_{n-1}(x)) 
&= \sum_{j=1}^{n-1} \sum_{i=0}^{j-1} (-1)^{i+j}
(a_0,..,a_ia_{i+1},.., a_j, g, a_{j+1},..,a_{n-1}) \\
 & \qquad + \sum_{j=0}^{n-1} (a_0,..,a_jg, a_{j+1},..,a_{n-1}) \\
 & \qquad - \sum_{j=0}^{n-2} (a_0,..,a_j, ga_{j+1},..,a_{n-1}) \\
 & \qquad + \sum_{j=0}^{n-3} \sum_{i=j+2}^{n-1} 
(-1)^{i+j} (a_0,..,a_j, g, a_{j+1},.., a_{i-1}a_i, ..,a_{n-1}) \\
 & \qquad + \sum_{j=0}^{n-1} (-1)^{n+j} 
(a_0,..,a_j,g, a_{j+1},..,a_{n-2}).
\end{align*}

Note that in the last term the summand for $j=n-1$ takes the form
$-(a_0,a_1,...,a_{n-1})$. 
Since $g$ is central, the third sum cancels against the second sum for 
$j\geq$ $1$. Also, we may reindex the second double sum and replace 
$i$ by $i-1$. This yields

\begin{align*}
\delta_n(f_n(x)) 
&= \sum_{j=1}^{n-1} \sum_{i=0}^{j-1} (-1)^{i+j}
(a_0,..,a_ia_{i+1},.., a_j, g, a_{j+1},..,a_{n-1}) \\
 & \qquad + (a_0g,a_1,...,a_{n-1}) \\
 & \qquad + \sum_{j=0}^{n-3} \sum_{i=j+1}^{n-2} 
(-1)^{i+j+1} (a_0,..,a_j, g, a_{j+1},.., a_{i}a_{i+1}, ..,a_{n-1}) \\
 & \qquad + \sum_{j=0}^{n-1} (-1)^{n+j} 
(a_0,..,a_j, g, a_{j+1},..,a_{n-2}),
\end{align*}
where as before the summand for $j=n-1$ in the last sum is $-(a_0,a_1,...,a_{n-1})$.
We need to show that this is equal to the expression

\begin{align*} (g-1)x - s_{n-2}(\delta_{n-1}(x))
&= (ga_0,a_1,..,a_{n-1})-(a_0,..,a_{n-1}) \\
& \qquad 
- \sum_{i=0}^{n-2} (-1)^i s_{n-2}(a_0,..,a_ia_{i+1},..,a_{n-1})\\
& \qquad - (-1)^{n-1} s_{n-2}(a_0,..,a_{n-2}).
\end{align*}

Expanding $s_{n-2}$ and adjusting indexing yields that this is equal to 

\begin{align*}
& (ga_0,a_1,..,a_{n-1})  -(a_0,..,a_{n-1}) \\
& \qquad + \sum_{i=0}^{n-2}\sum_{j=0}^{i-1} (-1)^{i+j+1} 
(a_0,..,a_j, g, a_{j+1},.., a_ia_{i+1},.., a_{n-1}) \\
& \qquad +\sum_{i=0}^{n-2}\sum_{j=i+1}^{n-1} (-1)^{i+j} 
(a_0,..,a_ia_{i+1},..,a_{j}, g, a_{j+1},..,a_{n-1})\\
& \qquad + \sum_{j=0}^{n-2} (-1)^{n+j} 
(a_0,..,a_j, g, a_{j+1},..,a_{n-2}).
\end{align*}

We need to match all summands to those for the expression
of $\delta_n(s_{n-1}(x))$. The first summand cancels against the
second summand in $\delta_n(s_{n-1}(x))$. The second summand cancels
agains the summand for $j=n-1$ in the fourth (and last) sum of
$\delta_n(s_{n-1}(x))$. The last sum cancels against the remaining
summands of the fourth sum in $\delta_n(s_{n-1}(x))$.
The first double sum cancels against the second double sum in
$\delta_n(s_{n-1}(x))$ because both can be written as sum indexed by
pairs $(i,j)$ with $0\leq i<j\leq n-2$. Similarly, 
the remaining double sum cancels against
the first double sum in $\delta_n(s_{n-1}(x))$ because both can be
written as a sum indexed by pairs $(i,j)$ with $0\leq i<j \leq n-1$.
\end{proof}

\section{The BV operator in degree $2$}\label{appSection}

For convenience, we restate Theorem \ref{th:main} (xi). 
For $x$, $y$ elements in a group $G$ we write $[x,y]$ for 
$xyx^{-1}y^{-1}$, and we recall that $[x,yz]=[x,y]\,y[x,z]y^{-1}$. 

\begin{theorem}\label{th:DeltaH2}
Let $G$ be a finite group, $g\in Z(G)$, and $x\in$ $H^2(G,k)$. 
Suppose that $x$ corresponds to a central extension
\[ 1 \to k^+ \to  K \to G \to 1. \]
For any  $h\in G$, choose an inverse image $\hat h \in K$. 
Then identifying $H^1(G,k)$ with $\Hom(G,k)$, we have
\[ [\hat g,\hat h] = \Delta_g(x)(h) \in k^+. \]
\end{theorem}

We combine Theorem \ref{thm:Deltag3} and the following Lemma to give 
a proof of Theorem \ref{th:DeltaH2}.

\begin{lemma} \label{lem:2cocycleswitch}
Let $1\to Z\to \hat G\to G\to 1$ be a central extension of an
abelian group $G$. For $g\in$ $G$ choose an inverse image $\hat g$ of 
$g$ in $\hat G$ such that $\hat{1}_G=$ $1_{\hat G}$. Let 
$\alpha \colon G\times G\to$ $Z$ be the $2$-cocycle defined by 
$$\alpha(g,h)= \hat g\hat h \widehat{gh}^{-1}$$
for $g$, $h\in G$. Let $\beta, \gamma \colon G\times G\to$ $Z$ be
the maps defined by
\begin{gather*}
\beta(g,h) = \alpha(h,g),\\
\gamma(g,h) = [\hat g, \hat h] = 
\hat g\hat h \hat g^{-1}\hat h^{-1}
\end{gather*}
for $g$, $h\in G$. Then $\beta$, $\gamma$ are $2$-cocycles, and
we have $\alpha\beta^{-1} = \gamma$.
\end{lemma}

\begin{proof}
All parts are trivial verifications.
\end{proof}

\begin{proof}[{Proof of Theorem \ref{th:DeltaH2}}]
By  Theorem~\ref{th:main}\,(iii),  $\Delta_g$ commutes with restriction to  subgroups of $G$ 
containing $g$. Hence, it suffices to consider the case where
$G=\langle g,h\rangle$. Therefore, in order to prove the formula for 
$\Delta_g(x)(h)$ we may assume that $G$ is abelian. 
Note that with the notation of Theorem \ref{thm:Deltag3} we have 
$$s_1(a_0,a_1) = (a_0,g, a_1) - (a_0,a_1,g)$$
and hence its dual $(s_1)^\vee$ sends a $kG$-homomorphism
$\zeta \colon kG^3\to$ $k$ to the $kG$-homomorphism 
$\zeta\circ s_1 \colon kG^2\to$ $k$ sending $(a_0,a_1)$ to
$\zeta(a_0, g, a_1) - \zeta(a_0,a_1,g)$. The identification of 
$H^2(G;k)$ in terms of classes of cocycles is given via the adjunction 
map 
$$\Hom_{kG}(kG^3,k)\cong \Hom_k(kG^2,k)$$ 
sending a $kG$-homomorphism $\zeta \colon kG^3\to$ $k$ to the $k$-linear map 
$kG^2\to$ $k$ determined by the assignment $(a_1,a_2)\mapsto$ 
$\zeta(1,a_1,a_2)$. Similarly, the identification $H^1(G,k)=$ 
$\Hom(G,k^+)$ is given via the adjunction $\Hom_{kG}(kG^2,k)\cong$ 
$\Hom_k(kG,k)$. Thus if $\alpha$ is a $2$-cocyle representing the class 
$x$, then the $1$-class determined by $\alpha\circ s_1$ is given by the 
assignment
$$a_1 \mapsto \alpha(g, a_1) - \alpha(a_1, g)$$
By Lemma \ref{lem:2cocycleswitch} applied with $Z=k^+$ (written
additively) this is equal to the map 
$$a_1 \mapsto [\hat a_1, \hat g] , $$
and writing $h$ instead of $a_1$, we get that 
$\Delta_g(x)(h)=[\hat h, \hat g]$ for all $h\in$ $G$ as claimed.
\end{proof}

\section{Examples}\label{ExamplesSection}

Throughout this section $p$ is a prime and $k$ is a field of 
characteristic $p$. As a consequence of Theorem~\ref{th:main}\,(vii), in
order to calculate the maps $\Delta_g$, where $g$ is a central 
element in a finite group $G$, it suffices to calculate in the
case where $G$ is a finite $p$-group. See \cite{Adem/Milgram:1994a} for
background material.

\begin{eg} \label{cyclic-Example}
Let $G$ be a finite cyclic $p$-group and let $g\in$ $G$.
Then $H^*(G,k)$ contains a polynomial subalgebra $k[x]$ with
$x$ in degree $2$, such that $H^*(G,k)$ is generated as a module over 
$k[x]$ by $1$ and a degree $1$ element $y$. Then $x$ is in the 
image of the map $H^2(G,\bZ)\to$ $H^*(G,k)$, and hence 
$\Delta_g(x)=0$ by
Theorem~\ref{th:main}\,(x). Moreover, by
Theorem~\ref{th:main}\,(ix)
we have $\Delta_g(y)=$ $y(g)$. Thus,
using that $\Delta_g$ is a derivation by Theorem~\ref{th:main}\,(i), 
for any nonnegative integer $n$, we have $\Delta_g(x^n)=0$ and
$\Delta_g(x^ny)=$ $x^n y(g)$. 

Using the canonical identification $\HH^*(kG)=$ 
$kG\tenk H^*(G,k)$ from Holm~\cite{Holm:1996a}, Cibils and 
Solotar~\cite{Cibils/Solotar:1997a}, this determines the BV operator 
$\Delta$ on $\HH^*(kG)$ as follows: we have 
\begin{gather*}
\Delta(g \ten x^n) = 0, \\
\Delta(g\ten x^ny) = y(g)\cdot g \ten x^n.
\end{gather*}
Other papers dealing with BV and Gerstenhaber structures on 
Hochschild cohomology of cyclic groups
include S\'anchez-Flores~\cite{SanchezFlores:2012a},
Yang~\cite{Yang:2013a},
Angel and Duarte~\cite{Angel/Duarte:2017a}.
\end{eg}

\begin{eg} \label{abelian-Example}
Let $G$ be a finite abelian $p$-group, and let $g\in$ $G$. Combining 
Example~\ref{cyclic-Example} with the K\"unneth formula
Proposition \ref{kunneth-prop2} determines $\Delta_g$. Using again
the canonical identification of algebras $\HH^*(kG)=kG\tenk H^*(G,k)$
from \cite{Holm:1996a}, \cite{Cibils/Solotar:1997a}, the BV operator 
$\Delta$ on $\HH^*(kG)$ is determined by 
$\Delta(g \ten \zeta) = g \ten \Delta_g(\zeta)$.
\end{eg}

\begin{eg}\label{eg:quaternion}
Let $G$ be the generalised quaternion group $Q_{2^n}$ 
of order $2^n$ $(n\ge 3)$, and let $\gamma$ be the
central element of order two in $G$. Then
$H^*(G,\bF_2)$, is generated by two elements $x$ and $y$ in
degree one and an element $z$ of degree four. The elements $x$
and $y$ are nilpotent, and generate a finite Poincar\'e duality
algebra with top degree three. Since $z$ is not a zero divisor, we have
\[ \sum_{i=0}^\infty t^i\dim_{\bF_2}H^i(Q_{2^n},\bF_2)
= \frac{1+2t+2t^2 +t^3}{1-t^4}. \]
The exact relations between $x$ and $y$ depend on whether $n=3$
or $n>3$. If $n=3$ then $x^2+xy+y^2=0$ and $x^2y+xy^2=0$,
whereas if $n>3$ then the relations are $xy=0$ and $x^3+y^3=0$.
Again using 
Theorem~\ref{th:main}\,(ix), and using the fact that $\gamma$ belongs to
the derived subgroup of $G$,
we have
$\Delta_\gamma(x)=\Delta_\gamma(y)=0$, and it remains to
compute $\Delta_\gamma(z)$.

In both cases, $H^3(G,\bF_2)$ is one dimensional, and
we have $H^3(G,\bZ)=0$. Moreover, there is an element
$z'\in H^4(G,\bZ)$ with $2^nz'=0$, and such that $z$
is the image of $z'$ under the reduction mod two map
$H^4(G,\bZ) \to H^4(G,\bF_2)$. Therefore $\Delta_\gamma(z')=0$,
and the commutativity of the diagram
\[ \xymatrix{H^4(G,\bZ) \ar[r]^{\Delta_\gamma} \ar[d] &
H^3(G,\bZ) \ar[d] \\
H^4(G,\bF_2) \ar[r]^{\Delta_\gamma}\ar[r] & H^3(G,\bF_2)} \]
shows that $\Delta_\gamma(z)=0$. 
Since $\Delta_\gamma$ is a derivation, it follows that
it is zero on all elements of $H^*(G,\bF_2)$.

Other papers considering the BV structure on Hochschild cohomology of 
quaternion groups include Ivanov, Ivanov, Volkov and
Zhou~\cite{Ivanov/Ivanov/Volkov/Zhou:2015a},
Ivanov~\cite{Ivanov:2016a}.
\end{eg}

\begin{eg}\label{eg:dihedral}
Let $G$ be the dihedral group of order $2^n$ with $n\ge 3$,
\[ G=\langle g,h \mid g^2=h^2=1,\ (gh)^{2^{n-1}}=1\rangle, \]
and let $\gamma$ be the central 
involution $(gh)^{2^{n-2}}$ in $G$. Then $H^*(G,\bF_2)=\bF_2[x,y,z]/(xy)$, where 
$|x|=|y|=1$ and $|z|=2$. Using 
Theorem~\ref{th:main}\,(ix) and the fact that $\gamma$ belongs to
the derived subgroup of $G$,
we have $\Delta_\gamma(x)=\Delta_\gamma(y)=0$, so it remains to
evaluate $\Delta_\gamma(z)$. Now $z$ corresponds to the
central extension
\[ 1 \to \bF_2^+ \to Q \to G \to 1 \]
where $Q$ is the generalised quaternion group of order $2^{n+1}$
\[ Q=\langle \hat g,\hat h\mid \hat g^2 =\hat h^2=(\hat g \hat h)^{2^{n-1}},
  (\hat g\hat h)^{2^n}=1\rangle. \]
The inverse image of $\gamma$ in $Q$ is not central, so it follows from
Theorem~\ref{th:main}\,(xi) that $\Delta_\gamma(z)\ne 0$. Now
$G$ has an automorphism of order two swapping $g$ and $h$,
and swapping $x$ and $y$. It lifts to an automorphism of
$Q$ swapping $\hat g$ and $\hat h$, and therefore fixes $z$.
So it also fixes $\Delta_\gamma(z)$. The only non-zero fixed
element of degree one is $x+y$, and therefore we have 
\[ \Delta_\gamma(z)=x+y. \]
Since $\Delta_\gamma$ is a derivation, this determines its
value on all elements of $H^*(G,\bF_2)$.  We have
\[ \Delta_\gamma(x^iz^j)=jx^{i+1}z^{j-1},\qquad
\Delta_\gamma(y^iz^j)=jy^{i+1}z^{j-1}. \]
\end{eg}

\begin{remark}
As well as the proof given in  Section~\ref{appSection}, Theorem~\ref{th:main}\,(xi)  
can be proved by a direct computation  as follows. 
Note that $[\hat g, \hat h]\in$ $k^+$ is central, hence equal to
$[\hat h^{-1}, \hat g]$, and the formula for $\Delta_g(x)$ as stated 
does indeed define a group homomorphism from $G$ to $k^+$. 
Next, since $\Delta_g$ commutes with restriction to 
subgroups of $G$ containing $g$, it suffices to consider the case where
$G=\langle g,h\rangle$. This is an abelian group, and so by Example
\ref{abelian-Example}, in principle we are done. If $H^1(G,k)$ is one 
dimensional, both sides are zero. Otherwise, it is two dimensional, 
and $H^2(G,k)$ modulo the image of $H^2(G,\bZ)\to H^2(G,k)$ is one 
dimensional, spanned by the product of two degree one elements. This 
reduces the proof to an explicit and slightly tedious computation.
\end{remark}

\begin{eg}\label{eg:semidihedral}
Let $G$ be the semidihedral group of order $2^n$ ($n\ge 4$). This is
the group
\[ G = \langle g,h\mid g^2=1, h^{2^{n-1}}=1, ghg =
  h^{2^{n-2}-1}\rangle. \]
The cohomology ring was computed by Munkholm (unpublished) and 
Evens and Priddy~\cite{Evens/Priddy:1985a}. We have
\[ H^*(G,\bF_2) =\bF_2[x,y,z,w]/(x^3,xy,xz,z^2+y^2w), \]
where $|x|=|y|=1$, $|z|=3$, and $|w|=4$.
Evens and Priddy also observed that the cohomology is detected on
the dihedral subgroup of order eight, $D=\langle g,h^{2^{n-3}}\rangle$
and the quaternion subgroup of order eight, $Q=\langle gh,h^{2^{n-3}}\rangle$.
Let $\gamma$ be the central element $h^{2^{n-2}}$ of order two. Using 
Theorem~\ref{th:main}\,(ix) and the fact that $\gamma$ belongs to
the derived subgroup of $G$,
we have $\Delta_\gamma(x)=\Delta_\gamma(y)=0$.

Applying $\Delta_\gamma$ to the relation $z^2=y^2w$ and using the
fact that $\Delta_\gamma$ is a derivation, we have
$y^2\Delta_\gamma(w)=0$. Now multiplication by $y^2$ from
$H^2(G,\bF_2)$ to $H^4(G,\bF_2)$ is injective, so we deduce that
$\Delta_\gamma(w)=0$. 

It remains to compute $\Delta_\gamma(z)$. This has degree two, so
it is a linear combination of $x^2$ and $y^2$. To determine
which, we use the information at the bottom of page~71
of~\cite{Evens/Priddy:1985a} on restriction to $D$ and $Q$.
First note that our $x$ and $y$ are their $x$ and $x+y$;
this is determined by which non-zero element of $H^1(G,\bF_2)$ is nilpotent
and what it annihilates.
So $y$ restricts to
zero on $Q$, while $x$ and $x^2$ have non-zero restriction. Since $\Delta_\gamma$
is zero on $H^*(Q,\bF_2)$ by Example~\ref{eg:quaternion},
it follows that $\Delta_\gamma(z)$ is a multiple of $y^2$.
The restriction of $z$ to $H^*(D,\bF_2)$ 
is a degree three element which is not in the subring
generated by $H^1(D,\bF_2)$. So using Example~\ref{eg:dihedral},
it follows that $\Delta_\gamma$ is non-zero on the restriction
of $z$, and hence $\Delta_\gamma(z)$ cannot be zero.
Hence we have $\Delta_\gamma(z)=y^2$.

Using the fact that $\Delta_\gamma$ is a derivation, it
is determined by the information that $\Delta_\gamma(x)=0$, 
$\Delta_\gamma(y)=0$, $\Delta_\gamma(z)=y^2$, $\Delta_\gamma(w)=0$.
\end{eg}

\section{The Gerstenhaber bracket}

Throughout this section, let $p$ be a prime and let $k$ be a field of 
characteristic $p$.
In this section,  we combine the formula
\begin{equation}\label{eq:Gerstenhaber} 
[x,y]=(-1)^{|x|}\Delta(xy)-(-1)^{|x|}\Delta(x)y 
-x\Delta(y)
\end{equation}
relating the BV operator $\Delta$ to the Gerstenhaber bracket
and products in $\HH^*(kG)$,
with the Siegel--Witherspoon formula
\begin{equation}\label{eq:SW} 
xy=\sum_u\Tr^{C_G(guhu^{-1})}_{C_G(g)\cap C_G(uhu^{-1})}
(\Res^{C_G(g)}_{C_G(g)\cap C_G(uhu^{-1})}(x)
\cdot \Res^{C_G(uhu^{-1})}_{C_G(g)\cap C_G(uhu^{-1})}(u^*(y))) 
\end{equation}
for products in $\HH^*(kG)$ in terms of the centraliser 
decomposition from \cite{Siegel/Witherspoon:1999a}. Here, $G$ is a 
finite group, $g$, $h\in$ $G$, $x\in H^*(C_G(g),k)$ and 
$y\in H^*(C_G(h),k)$, and $u$ runs over a set of double coset 
representatives  of $C_G(g)$ and $C_G(h)$ in $G$.
The notation $u^*(y)$ denotes the image of $y$ in $H^*(C_G(uhu^{-1}),k)$ under
conjugation by $u$.
The left side is the product of $x$ and $y$ regarded as elements in
$\HH^*(kG)$ via the centraliser decomposition.
The summand on the right hand side indexed by $u$ is in the summand in 
the centraliser decomposition corresponding to $guhu^{-1}$, and the
multiplication of the restrictions on the right is the usual cup product 
in $H^*(C_G(g)\cap C_G(guhu^{-1}),k)$. 

Combining~\eqref{eq:Gerstenhaber} and~\eqref{eq:SW}, we have
\begin{multline}\label{eq:[x,y]}
[x,y] = \\
\sum_u\Bigl((-1)^{|x|}\Delta_{guhu^{-1}}\Tr^{C_G(guhu^{-1})}_{C_G(g)\cap C_G(uhu^{-1})}
\bigl(\Res^{C_G(g)}_{C_G(g)\cap C_G(uhu^{-1})}(x)
\cdot \Res^{C_G(uhu^{-1})}_{C_G(g)\cap C_G(uhu^{-1})}(u^*(y))\bigr) 
\\
+\Tr^{C_G(guhu^{-1})}_{C_G(g)\cap C_G(uhu^{-1})}\bigl( -(-1)^{|x|}
\Res^{C_G(g)}_{C_G(g)\cap  C_G(uhu^{-1})}(
\Delta_g(x))\cdot\Res^{C_G(uhu^{-1})}_{C_G(g)\cap
  C_G(uhu^{-1})}(u^*(y))
\\
-\Res^{C_G(g)}_{C_G(g)\cap C_G(uhu^{-1})}(x)\cdot
\Res^{C_G(uhu^{-1})}_{C_G(g)\cap C_G(uhu^{-1})}(u^*(\Delta_h(y)))
\bigr) \Bigr).
\end{multline}

Note that if $g\in Z(G)$ then this formula simplifies
considerably. Namely, $C_G(g)=G$, $C_G(g)\cap C_G(h)=C_G(gh)$, and
the transfers do not do anything. There is only one
double coset, and we may take $u=1$. The formula then becomes
\[ [x,y]=(-1)^{|x|}\Delta_{gh}(\Res^G_{C_G(gh)}(x) \cdot y) 
- (-1)^{|x|}\Res^G_{C_G(gh)}(\Delta_g(x))\cdot y
- \Res^G_{C_G(gh)}(x)\cdot \Delta_h(y), \]
as an element of $H^*(C_G(gh),k)$ in the centraliser decomposition.
Using Theorem~\ref{th:main}\,(i) and~(v), we expand the first term on the right side
into four terms.
Then using Theorem~\ref{th:main} (iii), two  of these cancel with the remaining two
terms to leave
\begin{equation}\label{eq:[x,y]6} 
[x,y] = (-1)^{|x|}\Delta_h(\Res^G_{C_G(gh)}(x))\cdot y
+ \Res^G_{C_G(gh)}(x)\cdot \Delta_g(y).
\end{equation}
If $x$ and $y$ have degree one then using
Theorem~\ref{th:main}\,(ix), this formula becomes
\begin{equation}\label{eq:[x,y]5}
[ x,y]= -x(h)y + y(g)\Res^G_{C_G(gh)}(x). 
\end{equation}
Now assume that $G$ is a finite $p$-group. 
If $g\in Z(G)\cap \Phi(G)$, where $\Phi(G)$ is the Frattini subgroup of $G$,  
then we have $y(g)=0$, and this simplifies
further to
\[ [x,y]=-x(h)y. \]
Again, these formulas are as elements of the $H^*(C_G(gh),k)$
component in the centraliser decomposition. 

If both $g$ and $h$ are in $Z(G)$, then
notationally, it helps to keep track of $g$ and $h$
by writing $g\otimes x$ and $h\otimes y$, since the subring
of $\HH^*(kG)$ corresponding to elements in the centre in 
the centraliser decomposition is isomorphic to $kZ(G) \otimes 
H^*(G,k)$. With this notation, equation~\eqref{eq:[x,y]6} becomes
\[ [g\otimes x,h\otimes y]=
gh\otimes ( (-1)^{|x|}\Delta_h(x)\cdot y + x\cdot\Delta_g(y)). \]
In particular, $kZ(G)\otimes H^*(G,k) \subseteq \HH^*(kG)$ is a Lie subalgebra.
Finally, if $x$ and $y$ have degree one, this becomes
\begin{equation}\label{eq:[x,y]7} 
[g\otimes x,h\otimes y]=gh\otimes(-x(h)y+y(g)x).
\end{equation}
If $g$, $h\in Z(G)\cap \Phi(G)$ then the terms $x(h)$ and $y(g)$ vanish, and
the Lie bracket is equal to zero.

We record these observations in the following proposition.

\begin{prop}\label{pr:Z(G)leqPhi(G)}
Let $G$ be a finite $p$-group.
Suppose that $g\in Z(G)\cap \Phi(G)$ and $h\in G$. If $x\in
H^1(C_G(g),k)$ and $y\in H^1(C_G(h),k)$ in the centraliser
decomposition of $\HH^1(kG)$ then
\[ [x,y]=-x(h)y \]
as an element of $H^1(C_G(gh),k)$ in the centraliser decomposition.

In particular,  the Lie bracket is identically zero
on 
\begin{equation*}
k(Z(G)\cap\Phi(G))\otimes H^1(G,k) \leq \HH^1(kG). 
\qedhere
\end{equation*}
\end{prop}

\begin{remark}
If $g\in Z(G)\cap \Phi(G)$ and $h$ is not in $\Phi(G)$, then we can choose
$x\in H^1(C_G(g),k)$ such that $x(h)=-1$, and then for all $y\in
H^1(C_G(h),k)$ the element $[x,y]$ is $y$, but as an element of
$H^1(C_G(gh),k)$ in the centraliser decomposition. 

Taking $g=1$,
we see that given $h\not\in \Phi(G)$, there exists $x\in H^1(C_G(1),k)$ such that for 
every  $y\in H^1(C_G(h),k)$ we have $[x,y]=y$.
This proves in particular the following result.
\end{remark}

\begin{prop}
Let $G$ be a non-trivial finite $p$-group. Then the Lie algebra
$\HH^1(kG)$ is not nilpotent.
\end{prop}

On the other  hand, we have the  following theorem, whose proof is 
modelled on the method of Jacobson~\cite{Jacobson:1943a}.

\begin{theorem}\label{th:Phi<Z}
Suppose that $G$ is a finite $p$-group such that 
$|Z(G):Z(G)\cap\Phi(G)|\ge 3$.
Then the Lie subalgebra $kZ(G)\otimes H^1(G,k) \subseteq\HH^1(kG)$ is
not soluble, and therefore nor is $\HH^1(kG)$.
\end{theorem}

\begin{proof}
We compute inside $kZ(G) \otimes H^*(G,k) \subseteq \HH^*(kG)$, 
as described above.
Since by assumption we have $|Z(G):Z(G)\cap\Phi(G)|\ge 3$, either $p$ is odd, or
$Z(G)/(Z(G)\cap\Phi(G))$ is non-cyclic. We treat these two
cases separately.

(i) Suppose that $p$ is odd.
Choose an element $g\in Z(G)\setminus \Phi(G)$, and choose $x\in H^1(G,k)$
with $x(g)=1$.  Set $e=g \otimes x$, $f=-g^{-1}\otimes x$,
and $h=-1\otimes 2x$. We have $g\ne g^{-1}$, so
these elements are linearly independent, and using~\eqref{eq:[x,y]7}
we have
\begin{align*}
[e,f]&= -1\otimes (x(g)x-x(g^{-1})x)=h, \\
[h,e]&=-g\otimes (x(1) 2x - 2x(g)x)=2e, \\
[h,f]&= g^{-1}\otimes (x(1)2x - 2x(g^{-1})x)=-2f.
\end{align*}
Thus $e$, $f$ and $h$ span a copy of the Lie algebra $\sltwo$ inside
$\HH^1(kG)$. Therefore it is not soluble.

(ii) Suppose that $Z(G)/(Z(G)\cap\Phi(G))$ is non-cyclic. Choose 
elements $g$ and $h$ in $Z(G)$ so that their images in 
$Z(G)/(Z(G)\cap\Phi(G))$ generate distinct cyclic subgroups. Choose 
elements $x$ and $y$ in $H^1(G,k)$ with $x(g)=1$, $y(g)=0$, $x(h)=0$, 
$y(h)=1$. Then we compute
\begin{align*}
[g^{-1}\otimes x,g\otimes y] &=1\otimes y &
[h^{-1}\otimes y,h\otimes x]&=1\otimes x \\
[g \otimes x,1\otimes x]&=g\otimes x &
[h\otimes y,1\otimes y]&=h\otimes y\\
[g\otimes y,1\otimes x]&=g\otimes y &
[h\otimes x,1\otimes y]&=h\otimes x \\
[g^{-1}\otimes x,1\otimes x]&=-g^{-1}\otimes x&
[h^{-1}\otimes y,1\otimes y]&=-h^{-1}\otimes y \\
[g^{-1}\otimes y,1\otimes x]&=-g^{-1}\otimes y&
[h^{-1}\otimes x,1\otimes y]&=-h^{-1}\otimes x.
\end{align*}
Thus, letting $U$ be the linear
span in $\HH^1(kG)$  of the elements appearing in these computations,
we see that $[U,U]\supseteq U$. 
It follows that any Lie subalgebra containing $U$ is not soluble, and
hence $kZ(G)\otimes H^1(G,k)$ and $\HH^1(kG)$ are not soluble.
\end{proof}

\begin{remark}
If $G$ is the cyclic group of order two and $k$ has characteristic two 
then $\HH^1(kG)$ is a Lie algebra of dimension two, and is therefore soluble.
This shows that the
condition $|Z(G):Z(G)\cap \Phi(G)|\ge 3$ cannot be weakened
to $|Z(G):Z(G)\cap \Phi(G)|\ge 2$.
\end{remark}

\section{Extraspecial $p$-groups}

As an illustration of the methods developed in this paper,
in this section we examine the Lie structure of $\HH^1(kG)$ when
$G$ is an extraspecial $p$-group. The methods also cover some
other $p$-groups of class two, so we formulate them in
more generality.
We begin with a lemma.

\begin{lemma}\label{le:PhiTr}
Suppose that $G$ is a finite $p$-group, and that $g$, $h\in G$ satisfy 
\[ \Phi(C_G(g))=\Phi(C_G(h))=\Phi(G), \] 
and $C_G(g)\cap C_G(h)$
is a proper subgroup of $C_G(gh)$. If $x\in H^*(C_G(g),k)$, $y\in
H^*(C_G(h),k)$ are elements of degree zero or one, then 
\[ \Tr_{C_G(g)\cap C_G(h)}^{C_G(gh)}(
\Res^{C_G(g)}_{C_G(g)\cap C_G(h)}(x)\cdot \Res^{C_G(h)}_{C_G(g)\cap C_G(h)}(y)) = 0. \]
\end{lemma}

\begin{proof}
Since $\Phi(C_G(g))=\Phi(G)$ we may write $x=\Res^G_{C_G(g)}(\hat x)$
with $\hat x \in H^*(G,k)$. Similarly, $y=\Res^G_{C_G(h)}(\hat y)$
with $\hat y \in H^*(G,k)$. Then
\begin{align*} 
\Tr^{C_G(gh)}_{C_G(g)\cap C_G(h)}&(\Res^{C_G(g)}_{C_G(g)\cap C_G(h)}(x) 
\cdot \Res^{C_G(h)}_{C_G(g)\cap C_G(h)}(y)) \\
&= \Tr^{C_G(gh)}_{C_G(g)\cap C_G(h)}\Res^G_{C_G(g)\cap C_G(h)}
(\hat x\cdot\hat  y) \\
&=|C_G(gh):C_G(g)\cap C_G(h)|\Res^G_{C_G(gh)}(\hat x \cdot \hat y)\\
&=0 
\end{align*}
since $|C_G(gh):C_G(g)\cap C_G(h)|$ is zero in $k$.
\end{proof}

\begin{hyp}\label{hyp:cent}
We suppose that $G$ is a finite $p$-group satisfying $\Phi(G)=Z(G)$,
and $g$, $h$ are elements of $G$ such that 
for all $u\in G$, and for all $x\in H^*(C_G(g),k)$, 
$y\in H^*(C_G(h),k)$ of degree zero or one, we have either
\begin{enumerate}
\item
$C_G(g)\cap C_G(uhu^{-1}) = C_G(guhu^{-1})$, or
\item
$\Tr_{C_G(g)\cap C_G(uhu^{-1})}^{C_G(guhu^{-1})}(
\Res^{C_G(g)}_{C_G(g)\cap C_G(uhu^{-1})}(x)\cdot 
\Res^{C_G(uhu^{-1})}_{C_G(g)\cap C_G(uhu^{-1})}(u^*(y))) = 0.$
\end{enumerate}
\end{hyp}

We remark that if $\Phi(C_G(g))=\Phi(C_G(h))=\Phi(G)=Z(G)$
then Hypothesis~\ref{hyp:cent} holds by Lemma~\ref{le:PhiTr}.

Suppose that this hypothesis holds,
let $x\in H^1(C_G(g),k)$ and $y\in H^1(C_G(h),k)$,
and let $u$ be  a double coset representative of  $C_G(g)$ and
$C_G(h) $ in $G$.
Suppose first that the containment 
$C_G(g)\cap C_G(uhu^{-1}) \leq C_G(guhu^{-1})$
is proper. 
Then by Hypothesis~\ref{hyp:cent}, the  
term corresponding to $u$  in the Siegel--Witherspoon
formula~\eqref{eq:SW} for $xy$ is zero.
The same argument holds for $\Delta_g(x)y$
and $x\Delta_h(y)$, and so the term corresponding to $u$
in~\eqref{eq:[x,y]} vanishes. 

So the formula~\eqref{eq:[x,y]}  becomes
\begin{multline}\label{eq:[x,y]2}
[x,y] = 
\sum_u\Bigl(-\Delta_{guhu^{-1}}
\bigl(\Res^{C_G(g)}_{C_G(guhu^{-1})}(x)
\cdot \Res^{C_G(uhu^{-1})}_{C_G(guhu^{-1})}(u^*(y))\bigr) 
\\
{}+\Res^{C_G(g)}_{C_G(guhu^{-1})}(
\Delta_g(x))\cdot\Res^{C_G(uhu^{-1})}_{C_G(guhu^{-1})}(u^*(y))
-\Res^{C_G(g)}_{C_G(guhu^{-1})}(x)\cdot
\Res^{C_G(uhu^{-1})}_{C_G(guhu^{-1})}(u^*(\Delta_h(y)))
 \Bigr),
\end{multline}
where $u$ runs over those double coset representatives with
$C_G(guhu^{-1})=C_G(g)\cap C_G(uhu^{-1})$.
For such a representative $u$, $g$ commutes with $uhu^{-1}$.
Since $uhu^{-1}h^{-1}\in [G,G] \leq Z(G)$, this implies
that $g$ commutes with all conjugates of $h$.
Now applying Theorem~\ref{th:main}\,(iii), we have
\begin{multline}\label{eq:[x,y]3}
[x,y] = 
\sum_u\Bigl(-\Delta_{guhu^{-1}}
\bigl(\Res^{C_G(g)}_{C_G(guhu^{-1})}(x)
\cdot \Res^{C_G(uhu^{-1})}_{C_G(guhu^{-1})}(u^*(y))\bigr) 
\\
{}+\Delta_g\Res^{C_G(g)}_{C_G(guhu^{-1})}(
x)\cdot\Res^{C_G(uhu^{-1})}_{C_G(guhu^{-1})}(u^*(y))
-\Res^{C_G(g)}_{C_G(guhu^{-1})}(x)\cdot
\Delta_{uhu^{-1}}\Res^{C_G(uhu^{-1})}_{C_G(guhu^{-1})}(u^*(y))
 \Bigr).
\end{multline}

Set $x'=\Res^{C_G(g)}_{C_G(guhu^{-1})}(x)$,
$y'=\Res^{C_G(uhu^{-1})}_{C_G(guhu^{-1})}(u^*(y))$. By
Theorem~\ref{th:main}\,(v), we have
 $\Delta_{guhu^{-1}}=\Delta_g+\Delta_{uhu^{-1}}$
on $H^*(C_G(guhu^{-1}),k)$.
Hence, using Theorem~\ref{th:main}\,(i) and (ix), the
component of $[x,y]$ coming from $u$ is
\begin{align}
\label{eq:[x,y]4}
-&\Delta_{guhu^{-1}}(x'y'){}+
\Delta_g(x')y' 
-x'\Delta_{uhu^{-1}}(y') \notag
\\
&= -(\Delta_g(x')y'+\Delta_{uhu^{-1}}(x')y' 
-x'\Delta_g(y') -x'\Delta_{uhu^{-1}}(y')) 
+\Delta_g(x')y'
-x'\Delta_{uhu^{-1}}(y') \notag
\\
&=x'\Delta_g(y')-\Delta_{uhu^{-1}}(x')y'  \\
&=y(u^{-1}gu)\Res^{C_G(g)}_{C_G(guhu^{-1})}(x)-x(uhu^{-1})\Res^{C_G(uhu^{-1})}_{C_G(guhu^{-1})}(u^*(y)) \notag
\end{align}
where the multiplications are ordinary cohomology cup products
performed inside the summand $H^*(C_G(guhu^{-1}),k)$ corresponding to $guhu^{-1}$
in the centraliser decomposition.

We partition a set of conjugacy class representatives in 
 $G$ into subsets $C_i$,
where $g$ is in $C_i$ if $|G:C_G(g)|=p^i$, and for $i\ge 0$ we set
\[ X_i = \bigoplus_{g \in C_i} H^*(C_G(g),k),\qquad 
Y_i = \bigoplus_{j\ge i}X_j. \]
Here, as usual, the direct sum is over $G$-conjugacy classes of 
$g\in C_i$. Then we have
\[ \HH^*(kG) = \bigoplus_{i= 0}^n X_i. \]
It follows from~\eqref{eq:[x,y]4} that if $i\le j$ then
$[X_i,X_j]\subseteq Y_j$, and hence $[Y_i,Y_j]\subseteq Y_j$.
So the $Y_i$ are Lie ideals, and $\HH^1(kG)$ is soluble if
and only if each of the $Y_i/Y_{i+1}$ is soluble.
Note that by Proposition~\ref{pr:Z(G)leqPhi(G)}, since $Z(G)=\Phi(G)$
we have $[X_0,X_0]=0$, and $[X_0,X_i]=X_i$ for $i>0$.

As an example, we apply these methods to extraspecial $p$-groups.

\begin{theorem}\label{th:extraspecial}
Let $G$ be an extraspecial $p$-group. 
Then $\HH^1(kG)$ is a soluble Lie algebra.
The derived  length is two, 
except in the case where $G$ has order $p^3$ 
and is isomorphic to $\bZ/p^2\rtimes \bZ/p$,
in which case it has derived length three.
\end{theorem}
\begin{proof}
Let $G$ be extraspecial, and let $Z=Z(G)=\Phi(G)$, a group of order $p$. 
Every centraliser in $G$ is either equal to $G$ or has index $p$ in
$G$. We divide into two cases. The second case deals with extraspecial
groups which are semidirect products 
$\bZ/p^2\rtimes \bZ/p$, and the first case covers all other
extraspecial groups. So in the first case, if $|G|=p^3$ then $G$ is
either an extraspecial group of exponent $p$ ($p$ odd), or
the quaternion group of order eight ($p=2$).\bigskip

\noindent
{\bf Case 1.} $G\not\cong \bZ/p^2\rtimes \bZ/p$. In this case,
we claim that 
Hypothesis~\ref{hyp:cent} holds for every $g$ and $h$ in
$G$.  If $|G|\ge p^5$ then for every $g$ and $h$ in $G$ we have
$\Phi(C_G(g))=\Phi(C_G(h))=\Phi(G)=Z(G)$, so by Lemma~\ref{le:PhiTr},
the hypothesis holds. On the other hand, if $|G|=p^3$ with $p$ odd,
and part (i) of the hypothesis does not hold, then neither $g$ nor $h$
is central, and we have
$|C_G(g)\cap C_G(uhu^{-1})|=p$ and $|C_G(guhu^{-1})|=p^2$.
So $C_G(guhu^{-1})$ is elementary abelian, and then the transfer map from
any proper subgroup 
is zero. Finally, if $G$ is a quaternion group $Q_8$ then
the restriction maps from subgroups of order four to $Z$ is zero.
This completes the proof that Hypothesis~\ref{hyp:cent} holds
for every $g$ and $h$ in $G$.

So we may apply the theory derived in this section. 
Using the notation above, we have $\HH^1(kG)=X_0\oplus X_1$,
and since $Z(G)=\Phi(G)$ we have $[X_0,X_0]=0$ and $[X_0,X_1]\leq
X_1$. It remains to examine $[X_1,X_1]$. Let $g$ and $h$ be
non-central elements of $G$, and let $x\in H^1(C_G(g),k)$ and $y\in
H^1(C_G(h),k)$.

If $C_G(g)\ne C_G(h)$ then
there are no double coset representatives $u$ 
satisfying $C_G(g)\cap C_G(uhu^{-1})=C_G(guhu^{-1})$ in the
formula~\eqref{eq:[x,y]2},  because the intersection has index $p^2$
in $G$, and so  $[x,y]=0$. 

On the other hand, if
$C_G(g)=C_G(h)$ then writing $C$ for their common value,
the double cosets are just cosets of $C$.
Choosing $v \in G\setminus C$, we may take the double coset
representatives to be $1,v,\dots,v^{p-1}$. Now $v$ does not commute
with $h$, so defining $z=vhv^{-1}h^{-1}$, we have $Z=\langle
z\rangle$. For some $2\le m\le p-1$ we have $z^{m-1}=vgv^{-1}g^{-1}$,
$z^{m}=vghv^{-1}(gh)^{-1}$.
So for $0\le i\le p-1$ we have $v^ihv^{-i}=hz^i$,
$v^igv^{-i}=gz^{(m-1)i}$, and $v^ighv^{-i}=ghz^{mi}$.
Choose $n$ with $mn$ congruent to one modulo $p$. Then
$v^{ni}ghv^{-ni}=ghz^i$. 

All contributions to $[x,y]$ are landing in the same summand
in the centraliser decomposition, but need conjugating to match
the elements being centralised. The contribution coming from
$v^i$ is
\[ y(v^{-i}gv^{i})\,x - x(v^ihv^{-i})\,(v^i)^*(y) \]
in the $C_G(gv^ihv^{-i})$ component, namely the
$C_G(ghz^i)=C_G(v^{ni}ghv^{-ni})$ component. So we must conjugate
to get
\[ y(v^{-i}gv^i)\,(v^{-ni})^*(x) -
  x(v^ihv^{-i})\,(v^{-(n-1)i})^*(y) \]
in the $C_G(gh)$ component. If $|G|>p^3$, we have
$Z\leq \Phi(C)$, and so $x$ and $y$ vanish on $Z$. So conjugating
by $v^i$ has no effect. It follows
that the above term is independent of $i$, and when we sum
from $i=0$ to $p-1$ we get zero. 

On the other hand, if $|G|=p^3$ then $x$ and $y$ need not
vanish on the central element $z$. In this case, we obtain
\begin{align*}
y(v^{-i}gv^i)&=y(gz^{-(m-1)i})=y(g) - (m-1)iy(z), \\
x(v^ihv^{-i})&=x(hz^i)=x(h) + ix(z).
\end{align*}
So for $\gamma\in C$ we have
\[ [x,y](\gamma)=\sum_{i=0}^{p-1}\Bigl((y(g)-(m-1)iy(z))(x(\gamma)+nirx(z)) 
- (x(h)+ix(z)) (y(\gamma)+(n-1)iry(z))\Bigr). \]
Since $p\ge 3$, the expressions $\sum_{i=0}^{p-1}1$, $\sum_{i=0}^{p-1}i$ and 
 are zero in $k$, and so all but the quadratic term vanish. 
If $p\ge 5$ the quadratic term vanishes too since
$\sum_{i=0}^{p-1}i^2=0$ in $k$, but when 
$p=3$ it equals $-1$, and we have
\begin{align*}
 [x,y](\gamma) &= (-1)(-(m-1)y(z)nrx(z) - x(z)(n-1)ry(z) \\
&= r((m-1)n+(n-1))x(z)y(z) \\
&=r(mn-1)x(z)y(z).
\end{align*}
This is equal to zero
since $m$ and $n$ are inverses modulo $p$. This completes the proof in
the case $|G|=p^3$, and we are done.\bigskip

\noindent
{\bf Case 2.} $G \cong \bZ/p^2\rtimes \bZ/p$.
In this case, the failure of Hypothesis~\ref{hyp:cent} comes from the
fact that transfer from $Z$ to a subgroup $\bZ/p^2$
is non-zero in degree one (but zero in degree two). 
Restriction in degree one from $\bZ/p^2$ to $Z$ is
zero, however, and transfer from $\bZ/p$ to
$(\bZ/p)^2$ is zero in all degrees, 
so for the analysis above to fail, we must restrict
a degree one element from $(\bZ/p)^2$ to $Z$ and then transfer to $\bZ/p^2$.
So we write $X_1=X_1'\oplus X_1''$, where $X_1'$ is the sum
of the terms in the centraliser decomposition of $\HH^1(kG)$ 
with $C_G(g)$ elementary
abelian of order $p^2$ and $X_1''$ the sum of the terms with $C_G(g)$
cyclic of order $p^2$. 

As usual we have $[X_0,X_0]=0$,
$[X_0,X_1']=X_1'$, $[X_0,X_1'']=X_1''$. Furthermore, if $g$,
$h$ and $gh$ have the same centraliser 
then the same argument as in Case (i) shows that $[x,y]=0$. 
So we may assume that $C_G(g)\cap C_G(h)=Z$ and
$|C_G(gh)|=p^2$.

If $p$ is odd, then there is a unique subgroup isomorphic to
$(\bZ/p)^2$ in $G$, and so if $C_G(g)\cong (\bZ/p)^2$ and
$C_G(h)\cong \bZ/p^2$ then $C_G(gh)\cong \bZ/p^2$. Hence
$[X_1',X_1']=0$,
$[X_1',X_1'']\le X_1''$, and $[X_1'',X_1'']=0$. So $\HH^1(kG)$ is
soluble of derived length three.

On the other hand, if $p=2$ then there is a unique subgroup isomorphic
to $\bZ/4$ in $G$, and so if $C_G(g)\cong (\bZ/2)^2$ and $C_G(h)\cong
\bZ/4$ then $C_G(gh)\cong (\bZ/2)^2$. This time we have
$[X_1',X_1']\le X_1''$ and $[X_1',X_1'']=0$, and $[X_1'',X_1'']=0$.
So $\HH^1(kG)$ is again soluble of derived length three.
\end{proof}

\begin{remark}
The cases of the extraspecial $2$-groups $D_8$ and $Q_8$ of order
eight were also considered in~\cite{Eisele/Raedschelders:2020a,
RubioyDegrassi/Schroll/Solotar:preprint}.
\end{remark}

\bibliographystyle{amsplain}
\bibliography{repcoh}

\end{document}